\documentclass[11pt,twoside]{article}
\usepackage[margin=1in]{geometry}

\usepackage{mathrsfs}
\usepackage{amsfonts,amsmath,dsfont,amssymb,amsthm,stmaryrd,bbm}

\usepackage{hyperref}
\usepackage[dvipsnames,svgnames,x11names,hyperref, enumerate]{xcolor}
\definecolor{blueish}{HTML}{007F99}
\definecolor{purpleish}{HTML}{72177A}
\hypersetup{
  colorlinks,
  citecolor=purpleish,
  filecolor=red,
  linkcolor=blueish,
  urlcolor=Blue
}

\usepackage{algorithm}
\usepackage{url}  %Required
\usepackage{verbatim}
\usepackage{color}
\usepackage{graphicx}  %Required
\newsavebox{\algbox}
\usepackage{mathtools}
\usepackage{scalerel,stackengine}
\renewcommand{\epsilon}{\varepsilon}

\newtheorem{conj}{Conjecture}
\newtheorem{thm}{Theorem}

\newtheorem{lemma}[conj]{Lemma}
\newtheorem{claim}{Claim}

\newtheorem{defn}[conj]{Definition}

\newcommand{\f}[1]{\boldsymbol{#1}}
\newcommand{\bb}[1]{\mathbb{#1}}
\newcommand{\ca}[1]{\mathcal{#1}}

\newcommand{\s}[1]{\mathsf{#1}}

\newcommand{\remove}[1]{}
\newcommand{\ip}[2]{\langle #1, #2 \rangle}

\usepackage[utf8]{inputenc} % allow utf-8 input
\usepackage[T1]{fontenc}    % use 8-bit T1 fonts
\usepackage{hyperref}       % hyperlinks
\usepackage{url}            % simple URL typesetting
\usepackage{booktabs}       % professional-quality tables
\usepackage{amsfonts}       % blackboard math symbols
\usepackage{nicefrac}       % compact symbols for 1/2, etc.
\usepackage{microtype}      % microtypography
\usepackage{xcolor}         % colors

\title{Lower Bounds on the Total Variation Distance Between Mixtures of Two Gaussians}

\author{Sami Davies\thanks{Northwestern University. \url{sami@northwestern.edu}} 
\and 
{
Arya Mazumdar\thanks{Halicio\v{g}lu Data Science Institute, UC San Diego. \url{arya@ucsd.edu}}}
\and
{Soumyabrata Pal\thanks{Computer Science Department, University of Massachusetts Amherst. \url{soumyabratap@umass.edu}}}
\and
{
Cyrus Rashtchian\thanks{Department of Computer Science \& Engineering, UC San Diego. \url{crashtchian@eng.ucsd.edu}}}
}

\begin{document}

\maketitle

\begin{abstract}%
Mixtures of high dimensional Gaussian distributions have been studied extensively in statistics and learning theory. 
While the total variation distance appears naturally in the sample complexity of distribution learning, it is analytically difficult to obtain tight lower bounds for mixtures. Exploiting  a connection between total variation  distance and the  characteristic function of the mixture, we provide fairly tight functional approximations. This enables us to derive new lower bounds on the total variation distance between two-component Gaussian mixtures with a shared covariance matrix.
\end{abstract}

% \begin{keywords}%
% Learning theory, mixture distributions, statistical bounds
% \end{keywords}

\section{Introduction}

Let $\mathcal{N}(\f{\mu}, \f{\Sigma})$ denote the $d$-dimensional Gaussian distribution with mean $\f{\mu} \in \mathbb{R}^d$ and positive definite covariance matrix $\f{\Sigma} \in \mathbb{R}^{d\times d}$. A $k$-component mixture of $d$-dimensional Gaussian distributions is a distribution of the form $f = \sum_{i=1}^k w_i \cdot \mathcal{N}(\f{\mu}_i, \f{\Sigma}_i).$ Such a mixture is defined by $k$ triples $\{(w_i, \f{\mu}_i, \f{\Sigma}_i)\}_{i=1}^k$, where $w_i \in \mathbb{R}^+$ with $\sum_{i=1}^k w_i = 1$ are the mixing weights, $\f{\mu}_i \in \mathbb{R}^d$ are the means, and $\f{\Sigma}_i \in \mathbb{R}^{d\times d}$ are the covariance matrices. Mixtures of Gaussian distributions have been studied intensively due to their broad applicability to statistical problems~\cite{arora2001learning, dasgupta1999learning, dasgupta2000two,  huber2004robust, kane2020robust,  moitra2018algorithmic, moitra2010settling, pearson1894contributions, titterington1985statistical}.

The variational distance (a.k.a., the total variation (TV) distance) between two distributions $f,f'$ with same sample space $\Omega$ and  sigma algebra $\mathcal{S}$ %from a class of distributions with sample space $\Omega$ is a fundamental statistical property of the said class and 
is defined as follows:
\begin{align*}
    \left|\left|f-f'\right|\right|_{\s{TV}} \triangleq   \sup_{\ca{A}\in \mathcal{S} }\Big(f(\ca{A})-f'(\ca{A})\Big).
\end{align*}
The minimum pairwise TV distance of a class of distributions appears naturally in the expressions of statistical error rates  related to the class, most notably in the Neyman-Pearson approach to  hypothesis testing~\cite{lehmann2006testing,neyman2020contributions}, as well as in the sample complexity results in density estimation~\cite{devroye2012combinatorial}. In particular, in these applications, a lower bound on the total variation distance between two candidate distributions is an essential part of the algorithm design and analysis.

Tight bounds are known for the total variation distance between single Gaussians; however, they have only recently been derived as closed form functions of the distribution parameters~\cite{barsov1987estimates, devroye2018total}.  The functional form of the TV distance bound is often much more useful in practice because it can be directly evaluated based on only the means and covariances of the distribution. This has opened up the door for new applications to a variety of areas, such as analyzing ReLU networks~\cite{wu2019learning}, distribution learning~\cite{ashtiani2020near, BakshiDHKKK20}, private distribution testing~\cite{bun2019private, canonne2019private}, and average-case reductions~\cite{brennan2019optimal}.

 Inspired by the wealth of applications for single Gaussian total variation bounds, we investigate deriving analogous results for mixtures with two components. As our main contribution, 
we complement the single Gaussian results and derive  tight lower bounds for 
% the case of 
pairs of mixtures containing two equally weighted Gaussians with shared variance. 
We also present our results in a closed form in terms of the gap between the component means and certain statistics of the covariance matrix. The total variation distance between two distributions can be upper bounded by other distances/divergences (e.g., KL divergence, Hellinger distance) that are easier to analyze.  
In contrast, it is a key challenge to develop ways to lower bound the total variation distance.
% especially between mixtures with more than one component. 
% While we do not solve this problem in general, 
%We derive nearly tight lower bounds for the  case of pairs of mixtures containing two equally weighted Gaussians with shared variance.
The shared variance case is important because it presents some of the key difficulties in parameter estimation and is widely studied~\cite{daskalakis2016ten, wu2020optimal}. For example, mean estimation with shared variance serves as a model for the sensor location estimation problem in wireless or physical networks~\cite{kontkanen2004topics, liu2007survey, van2000asymptotic}. 
% Our bounds are given in terms of the gap between the component means of the mixtures. 
% Our focus on two-component mixtures with shared covariance allows us to develop new tools for analyzing the total variation distance.

The lower bound on total variation distance can be applicable in several contexts. In binary hypothesis testing, it gives a sufficient condition to bound from above the total probability of error of the best test~\cite{moitra2018algorithmic}. Hypothesis testing in Gaussian mixture models has been of interest, cf.~\cite{aitkin1985estimation,chen2009hypothesis}. Furthermore, parameter learning in Gaussian mixture models is a core topic in density estimation~\cite{devroye2012combinatorial}. Our bound can provide a sufficient condition on the {\em learnability} of the class of two component Gaussian mixtures in terms of the precision of parameter recovery and gap between the component of mixtures. Indeed, performance of various density estimation techniques, such as the Scheff\'e estimator or the minimum distance estimator, depends crucially on a computable lower bound in total variation distance between candidate distributions~\cite{devroye2012combinatorial}. Furthermore, our lower bound implies that, for the class of distributions we consider, if a pair of distributions is close in variational distance, then the distributions have close parameters. This type of implication is integral to arguments in outlier-robust moment estimation algorithms and clustering~\cite{BakshiDHKKK20, hopkins2018mixture}.

We obtain lower bounds on the total variation distance by examining the  characteristic function of the mixture. This connection has been previously used in~\cite{krishnamurthy20a} in the context of mixture learning, but it 
required strict assumptions on the mixtures having discrete parameter values, i.e., Gaussians with means that belong to a scaled integer lattice.  It is not clear how to generalize their techniques to non-integer means. As a first step towards that generalization, we analyze unrestricted two-component one-dimensional mixtures by applying a novel and more direct analysis of the characteristic function.
Then, in the high-dimensional setting, we obtain a new TV distance lower bound by projecting and then using our one-dimensional result.
% While this is a straightforward approach
By carefully choosing and analyzing the one-dimensional projection (which depends on the mixtures), we exhibit nearly-tight bounds on the TV distance of $d$-dimensional mixtures for any $d \geq 1$.

\subsection{Results}

Let $\ca{F}$ be the set of all $d$-dimensional, two-component, equally weighted mixtures
$$\ca{F} = \left  \{f_{\f{\mu}_0,\f{\mu}_1} =\frac12 \ca{N}(\f{\mu}_0,\f{\Sigma})+\frac12\ca{N}(\f{\mu}_1,\f{\Sigma}) \mid    \f{\mu}_0,\f{\mu}_1 \in \bb{R}^d, \f{\Sigma }\in \mathbb{R}^{d \times d}\right \},$$
where $\f{\Sigma} \in \bb{R}^{d \times d}$ is a positive definite matrix. 
% Let $\lambda_{\f{\Sigma}}^{\star}$ denote the largest eigenvalue of the matrix $\f{\Sigma}$.  We denote the Mahalonobis norm as $\left|\left|\f{x}\right|\right|_{\f{\Sigma}}= \sqrt{\f{x}^{T}\f{\Sigma}^{-1}\f{x}}$.
When $d=1$, we use the notation $f_{\mu_0,\mu_1} \in \mathcal{F}$ and simply denote the variance as $\sigma^2 \in \mathbb{R}$. 
Our main result is the following nearly-tight lower bound on the TV distance between pairs of $d$-dimensional two-component mixtures with shared covariance.

\begin{thm}{\label{thm:main_high_dim_tv}}
For $f_{\f{\mu}_0,\f{\mu}_1},f_{\f{\mu}_0',\f{\mu}_1'} \in \ca{F}$, define sets
$S_1= \{\f{\mu}_1-\f{\mu}_0,\f{\mu}_1'-\f{\mu}_0'\}$, $S_2 = \{\f{\mu}_0'-\f{\mu}_0,\f{\mu}_1'-\f{\mu}_1\}$, $S_3 = \{\f{\mu}_0'-\f{\mu}_1,\f{\mu}_1'-\f{\mu}_0\}$  and vectors 
 $\f{v}_1 = \s{argmax}_{s \in S_1} || \f{s}||_2$, $\f{v}_2 = \s{argmax}_{s \in S_2} || \f{s}||_2$, $\f{v}_3 = \s{argmax}_{s \in S_3} || \f{s}||_2$. %, and $\delta_1,\delta_2$ be 
 Let $\lambda_{\f{\Sigma},\ca{U}} \triangleq \max_{\f{u}:\left|\left|\f{u}\right|\right|_2=1, \f{u} \in \ca{U}} \f{u}^{T}\f{\Sigma}\f{u}$ with $\ca{U}$ being the span of the vectors $\f{v}_1,\f{v}_2,\f{v}_3$.
If $\|\f{v}_1\|_2  \ge \min(\|\f{v}_2\|_2,\|\f{v}_3\|_2)/2$ and $\sqrt{\lambda_{\f{\Sigma},\ca{U}}}=\Omega(\left|\left|\f{v}_1\right|\right|_2)$, then
\[
\left|\left|f_{\f{\mu}_0,\f{\mu}_1}-f_{\f{\mu}_0',\f{\mu}_1'} \right|\right|_{\s{TV}} = \Omega\Big(\min\Big(1, \frac{\|\f{v}_1\|_2\min(\|\f{v}_2\|_2,\|\f{v}_3\|_2)}{\lambda_{\f{\Sigma},\ca{U}}}\Big)\Big),
\]
 and otherwise, we have that
%  for $\f{v} = \s{argmax}_{s \in S_2} \left|\left|\f{s} \right|\right|_2$,
$
\left|\left|f_{\f{\mu}_0,\f{\mu}_1}-f_{\f{\mu}_0',\f{\mu}_1'} \right|\right|_{\s{TV}} = \Omega\Big(\min\Big(1,  \min(\|\f{v}_2\|_2,\|\f{v}_3\|_2)/ \sqrt{\lambda_{\f{\Sigma},\ca{U}}} \Big)\Big).
$
\end{thm}

Notice from the definitions of $\f{v}_1,\f{v}_2,\f{v}_3$ that $\ca{U}$ is contained within the subspace spanned by the unknown mean vectors $\f{\mu}_0,\f{\mu}_1,\f{\mu}_0',\f{\mu}_1'$. Furthermore, $\lambda_{\f{\Sigma},\ca{U}}$ as defined in Theorem \ref{thm:main_high_dim_tv} can always be bounded from above by the largest eigenvalue of the matrix $\f{\Sigma}$, and as we will show in Section \ref{subsection:tight}, this upper bound  characterizes the TV distance between mixtures in several instances.

In some cases, it is simpler to work with $\f{z} = \f{\Sigma}^{-1/2}\f{x}$ instead of the original samples $\f{x}\in \bb{R}^d$. Note that if $\f{x}\sim \frac12 \ca{N}(\f{\mu}_0,\f{\Sigma})+\frac12\ca{N}(\f{\mu}_1,\f{\Sigma})$, then $\f{z}\sim \frac12 \ca{N}(\f{\Sigma}^{-1/2}\f{\mu}_0,\f{I})+\frac12\ca{N}(\f{\Sigma}^{-1/2}\f{\mu}_1,\f{I})$, for $\f{I}$ the $d$-dimensional identity matrix.
Overall, if we scale the distribution by $\f{\Sigma}^{-1/2}$, then by the invariance property of TV distance (see, for instance, Section 5.3 in \cite{devroye2012combinatorial}), Theorem~\ref{thm:main_high_dim_tv} implies the following.
For $f_{\f{\mu}_0,\f{\mu}_1},f_{\f{\mu}_0',\f{\mu}_1'} \in \ca{F}$ and $S_1, S_2, S_3$ as above,
the scaled vectors are
$\f{v}_i = \s{argmax}_{s \in S_i} || \f{\Sigma}^{-1/2}\f{s}||_2$, for $i \in [3]$.
% $\f{v}_2 = \s{argmax}_{s \in S_2} ||\f{\Sigma}^{-1/2} \f{s}||_2$, $\f{v}_3 = \s{argmax}_{s \in S_3} || \f{\Sigma}^{-1/2}\f{s}||_2$.
If $\|\f{\Sigma}^{-1/2} \f{v}_1\|_2  \ge \min(\| \f{\Sigma}^{-1/2} \f{v}_2\|_2,\|\f{\Sigma}^{-1/2}\f{v}_3\|_2)/2$ and $||\f{\Sigma}^{-1/2}\f{v}_1||_2=O(1)$, then
\[
\left|\left|f_{\f{\mu}_0,\f{\mu}_1}-f_{\f{\mu}_0',\f{\mu}_1'} \right|\right|_{\s{TV}} = \Omega\Big(\min\Big(1, \|\f{\Sigma}^{-1/2}\f{v}_1\|_{2}\min(\|\f{\Sigma}^{-1/2}\f{v}_2\|_{2},\|\f{\Sigma}^{-1/2}\f{v}_3\|_{2})\Big)\Big),
\]
 and otherwise, 
%  for $\f{v} = \s{argmax}_{s \in S_2} \left|\left|\f{s} \right|\right|_2$,
$
\left |\left|f_{\f{\mu}_0,\f{\mu}_1}-f_{\f{\mu}_0',\f{\mu}_1'} \right|\right|_{\s{TV}} = \Omega\left(\min\left(1,  \min(\|\f{\Sigma}^{-1/2}\f{v}_2\|_{2},\|\f{\Sigma}^{-1/2}\f{v}_3\|_{2}) \right)\right).
$

In  the special case of one component Gaussians, i.e., $\f{\mu}_0= \f{\mu}_1$ and $\f{\mu}_0' =\f{\mu}_1'$, we recover a result by Devroye et al.~(see the lower bound in \cite[Theorem 1.2]{devroye2018total}, setting $\f{\Sigma}_1 = \f{\Sigma}_2$). In the one-dimensional setting, our next theorem shows a novel lower bound on the total variation distance between any two distinct two-component one-dimensional Gaussian mixtures from $\cal{F}$.

 \begin{thm}{\label{thm:main1}}
Without loss of generality, for $f_{\mu_0,\mu_1},f_{\mu_0',\mu_1'} \in \ca{F}$, suppose
$\mu_0\le \min(\mu_1,\mu_0',\mu_1')$ and $\mu_0' \le \mu_1'$.
 Further, let $\delta_1 = \max\{|\mu_0-\mu_1|, |\mu_0'-\mu_1'|\} $ and
$\delta_2 = \max\{|\mu_0'-\mu_0|, |\mu_1-\mu_1'|\} $.
If $[\mu_0',\mu_1'] \subseteq [\mu_0,\mu_1]$ and $\sigma=\Omega(\delta_1)$,
then we have that
$$
    ||f_{\mu_0,\mu_1}-f_{\mu_0',\mu_1'}||_{\s{TV}} \ge \Omega(\min(1, \delta_1\delta_2 /\sigma^2)),
$$
and otherwise,
% i.e. $\mu_1 \neq  \max\{\mu_0,\mu_0',\mu_1',\mu_1\}$,
$
    ||f_{\mu_0,\mu_1}-f_{\mu_0',\mu_1'}||_{\s{TV}} \ge \Omega(\min(1, \delta_2 /\sigma)).
$
\end{thm}

\subsection{Related Work}

Let $\f{I}$ denote the $d$-dimensional identity matrix. Statistical distances between a pair of $k$-component $d$-dimensional Gaussian mixtures $f=\sum_{i=1}^{k}k^{-1}\ca{N}(\f{\mu}_i,\f{I})$ and $f'=\sum_{i=1}^{k}k^{-1}\ca{N}(\f{\mu}'_i,\f{I})$ with shared, known component covariance $\f{I}$ have been studied in \cite{doss2020optimal, wu2020optimal}. For a $k$-component Gaussian mixture $f=\sum_{i=1}^{k}k^{-1}\ca{N}(\f{\mu}_i,\f{I})$, let $M_{\ell}(f)=\sum_{i=1}^{\ell}k^{-1}\f{\mu}_i^{\otimes \ell}$ where $\f{x}^{\otimes \ell}$ is the $\ell$-wise tensor product of $\f{x}$. We denote the Kullback-Leibler divergence, Squared Hellinger divergence, and $\chi^2$-divergence of $f,f'$ by $\left|\left|f-f' \right|\right|_{\s{KL}}, \left|\left|f-f' \right|\right|_{\s{H}^2}$, and $\left|\left|f-f' \right|\right|_{\chi^2}$ respectively. We write $\left|\left|M\right|\right|_F$ to denote the Frobenius norm of the matrix $M$.
Prior work shows the following.

\begin{thm}[Theorem 4.2 in \cite{doss2020optimal}]\label{thm:prev}
Consider mixtures
$f=\sum_{i=1}^{k}k^{-1}\ca{N}(\f{\mu}_i , \f{I})$ and $f'=\sum_{i=1}^{k}k^{-1}\ca{N}(\f{\mu}'_i,\f{I})$ where $\left|\left|\f{\mu}_i\right|\right|_2\le R,\left|\left|\f{\mu}_i'\right|\right|_2\le R,$ for all $i \in [k]$ and constant $R\ge 0$. For any distance $D\in \{\s{H}^2,\s{KL},\chi^2\}$, we  have
$
    \|f-f'\|_{D} = \Theta\Big(\max_{\ell \le 2k-1} \|M_{\ell}(f)-M_{\ell}(f')\|_{F}^2\Big).
$
\end{thm}
\noindent
% For TV distance bounds, the above result is suboptimal. 
This bound alone does not give a guarantee for the TV distance.
However it is well-known that,
% on $\left|\left|f_{\f{\mu}_0,\f{\mu}_1}-f_{\f{\mu}_0',\f{\mu}_1'} \right|\right|_{\s{TV}}$ by using the identity 
\begin{align}\label{eq:hellinger}
    \left|\left|f_{\f{\mu}_0,\f{\mu}_1}-f_{\f{\mu}_0',\f{\mu}_1'} \right|\right|_{\s{TV}} \ge \left|\left|f_{\f{\mu}_0,\f{\mu}_1}-f_{\f{\mu}_0',\f{\mu}_1'} \right|\right|_{\s{H^2}}.
\end{align}
We can use this in conjunction with Theorem~\ref{thm:prev} to get a lower bound on TV distance, but it is suboptimal for many canonical instances.
For example, consider one-dimensional Gaussian mixtures 
\begin{align}\label{eq:example}
    f = \frac12 \ca{N}(u,1)+\frac12 \ca{N}(-u,1) \quad \text{and} \quad f' = \frac12 \ca{N}(2u,1)+\frac12 \ca{N}(-2u,1).
\end{align}
Using Eq.~(\ref{eq:hellinger}) and Theorem~\ref{thm:prev}, we have that $\left|\left|f-f' \right|\right|_{\s{TV}}=\Omega(u^4)$.  On the other hand, by using our result (Theorem \ref{thm:main1}), we obtain the improved bound  $\left|\left|f-f' \right|\right|_{\s{TV}}=\Omega(u^2)$. The improvement becomes more significant as $u$ becomes smaller. Also, the prior result in Theorem~\ref{thm:prev} assumes that  the means of the two mixtures $f,f'$ are contained in a ball of constant radius, limiting its applicability. 

The TV distance between Gaussian mixtures with two components when $d=1$ has been recently studied in the context of parameter estimation~\cite{feller2016weak, ho2016convergence, manole2020uniform,heinrich2018strong}. The TV distance guarantees in these papers are more general, as they do not need the component covariances to be same. However,
the results and their proofs are tailored towards the case when both the mixtures
have zero mean. They do not apply when considering the TV distance between two mixtures with distinct means.
% (since it is no longer possible to subtract a single number to make both mixtures zero mean).
Theorems~\ref{thm:main_high_dim_tv} and~\ref{thm:main1} hold for all mixtures with shared component variances, without assumptions on the means.  

Further, our bound can be tighter than these prior results, even in the case when the mixtures have zero mean.
Consider again the pair of mixtures $f,f'$ defined in Eq.~(\ref{eq:example}) above.
 In \cite{manole2020uniform,ho2016convergence}, the authors show that $\left|\left|f-f'\right|\right|_{\s{TV}} = \Omega(u^4)$; see, e.g., Eq.~(2.7) in \cite{manole2020uniform}. Notice that this is the same bound that can be recovered from Theorem \ref{thm:prev}, and
 as we mentioned before, this bound is loose. By using Theorem \ref{thm:main1}, we obtain the improved bound $\left|\left|f-f'\right|\right|_{\s{TV}} = \Omega(u^2)$. Now
 consider a more general pair of mixtures, where for $u,v \ge 0$, we define
 \begin{align}\label{eq:example2}
    f = \frac12 \ca{N}(u,1)+ \frac12  \ca{N}(-u,1) \quad \text{and} \quad f' =  \frac12  \ca{N}(v,1)+ \frac12  \ca{N}(-v,1).
\end{align}    
%For $v=2u$, we obtain the pair of mixtures defined in equation \ref{eq:example}. 
In \cite{feller2016weak} (see the proof of Lemma G.1 part (b)), the authors show that $\left|\left|f-f'\right|\right|_{\s{TV}} = \Omega((u-v)^2)$. Notice that for the previous example in Eq.~(\ref{eq:example}) with $v=2u$, the result in \cite{feller2016weak} leads to the bound $\left|\left|f-f'\right|\right|_{\s{TV}} = \Omega(u^2)$, which is the same bound that can be obtained from Theorem \ref{thm:main1}. However, for any small $\epsilon>0$, by setting $v=u+\epsilon$, we see that the bound in \cite{feller2016weak} reduces to $\left|\left|f-f'\right|\right|_{\s{TV}}=\Omega(\epsilon^2)$. On the other hand, by using Theorem \ref{thm:main1}, we obtain  $\left|\left|f-f'\right|\right|_{\s{TV}} = \Omega(u \cdot \epsilon+\epsilon^2)$. Whenever $u \gg \epsilon$, our result provides a much larger and tighter lower bound. On the other hand, whenever $u<\epsilon$, our bound coincides with that of \cite{feller2016weak}.

\subsection{Tightness of the TV distance bound}\label{subsection:tight}
Our bounds on the TV distance are tight up to constant factors.
For example, let $\f{u}\in \bb{R}^d$ be a $d$-dimensional vector satisfying $\left|\left|\f{u}\right|\right|_2 <1$. Consider the mixtures
$
    f = 0.5 \ca{N}(\f{u},\f{I})+0.5 \ca{N}(-\f{u},\f{I})$ and $f' = 0.5 \ca{N}(2\f{u},\f{I})+0.5 \ca{N}(-2\f{u},\f{I})$. Considering the notation of Theorem~\ref{thm:main_high_dim_tv}, we have $\f{v}_1 = 2\f{u}$ and $\f{v}_2 = \f{u}$, and the first bound in the theorem implies that $\left|\left|f-f'\right|\right|_{\s{TV}} \ge \Omega(\left|\left|\f{u}\right|\right|_2^2).$ On the other hand, we use the inequality $\left|\left|f-f'\right|\right|_{\s{TV}} \le \sqrt{2\left|\left|f-f'\right|\right|_{\s{H}^2}}$
in conjunction with Theorem~\ref{thm:prev}. In the notation of Theorem~\ref{thm:prev}, note that $M_{1}(f)-M_{1}(f') =0$, and we can upper bound the max over $\ell \in\{2,3\}$ by the sum of the two terms to say that
\begin{align*}
    \left|\left|f-f'\right|\right|_{\s{TV}} &\leq
    O\Big(\max_{\ell \in \{2,3\}} \|M_{\ell}(f)-M_{\ell}(f')\|_{F}^2\Big)\\
    &\leq O(\left|\left|\f{u}\otimes \f{u}\right|\right|_{F}+\left|\left|\f{u}\otimes \f{u}\otimes \f{u}\right|\right|_{F}) = O(\left|\left|\f{u}\right|\right|^2_{2}+\left|\left|\f{u}\right|\right|^3_{2}). 
\end{align*}
Since $\left|\left|\f{u}\right|\right|_2 <1$, we see that $\left|\left|\f{u}\right|\right|^2_{2}$ is the dominating term on the RHS, and  $\left|\left|f-f'\right|\right|_{\s{TV}}= \Theta(\left|\left|\f{u}\right|\right|^2_{2})$. As a result, our TV distance bound in Theorem~\ref{thm:main_high_dim_tv} is tight as a function of the means for this example.

 Our bounds are tight in other instances too. Consider the second parts of Theorems \ref{thm:main_high_dim_tv} and \ref{thm:main1} when samples are pre-multiplied with $\f{\Sigma}^{-1/2}$.
% in the one-dimensional case, this implies the setting where the means of one mixture are not in the interval defined by the means of the other mixture, i.e. $[\mu_0',\mu_1'] \not \subseteq [\mu_0,\mu_1]$.
Here, we can use the triangle inequality to derive a simple upper bound on the TV distance, 
\begin{align*}
    \left|\left|f_{\f{\mu}_0,\f{\mu}_1}-f_{\f{\mu}_0',\f{\mu}_1'} \right|\right|_{\s{TV}} \le \frac{1}{2}\min\Big(&\left|\left|\ca{N}(\f{\mu}_0,\f{\Sigma})-\ca{N}(\f{\mu}_0',\f{\Sigma})\right|\right|_{\s{TV}}+\left|\left|\ca{N}(\f{\mu}_1,\f{\Sigma})-\ca{N}(\f{\mu}_1',\f{\Sigma})\right|\right|_{\s{TV}},\\
    &\left|\left|\ca{N}(\f{\mu}_1,\f{\Sigma})-\ca{N}(\f{\mu}_0',\f{\Sigma})\right|\right|_{\s{TV}}+\left|\left|\ca{N}(\f{\mu}_0,\f{\Sigma})-\ca{N}(\f{\mu}_1',\f{\Sigma})\right|\right|_{\s{TV}}\Big).
\end{align*}
% The upper bound is the sum of the TV distance between single Gaussians. We can pair up the means of the two mixtures to minimize the upper bound. 
Then, we can use tight bounds on the TV distance between single Gaussians. In the one-dimensional setting, Theorem 1.3 in \cite{devroye2018total} shows that $||f_{\mu_0,\mu_1}-f_{\mu_0',\mu_1'}||_{\s{TV}} = O(\max(1,\delta_2/\sigma))$, recalling that $\delta_2 =  \max\{|\mu_0'-\mu_0|, |\mu_1-\mu_1'|\}$.  In the high dimensional setting, Theorem 1.2 in \cite{devroye2018total} shows that $$\left|\left|f_{\f{\mu}_0,\f{\mu}_1}-f_{\f{\mu}_0',\f{\mu}_1'}\right|\right|_{\s{TV}} = O\left(\max\left(1,\frac{1}{\lambda_{\min}(\f{\Sigma})} \cdot \min \left(\|\f{v_2}\|_2,\|\f{v_3}\|_2\right)\right)\right),$$ 
recalling $\f{v}_2$ and $\f{v}_3$ from the definitions in Theorem \ref{thm:main_high_dim_tv} and letting $\lambda_{\min}(\f{\Sigma})$ be the minimum eigenvalue of $\f{\Sigma}$. 
Again by the invariance property, TV distance remains the same if the samples are pre-multiplied by $\f{\Sigma}^{-1/2}$. With this transformation, the component co-variance matrix is $\f{I}$ and 
$$\left|\left|f_{\f{\mu}_0,\f{\mu}_1}-f_{\f{\mu}_0',\f{\mu}_1'}\right|\right|_{\s{TV}} = O\left(\max\left(1,\min\left(\|\f{\Sigma}^{-1/2}\f{v_2}\|_{2},\|\f{\Sigma}^{-1/2}\f{v_3}\|_{2}\right)\right)\right).$$ It follows that the second parts of Theorems \ref{thm:main_high_dim_tv} and \ref{thm:main1} are tight up to constants when samples are pre-multiplied with $\f{\Sigma}^{-1/2}$.

On the other hand, when comparing with the Hellinger distance upper bound of \cite{doss2020optimal}, our lower bound on the TV distance is not tight in the following case. Define two $d$-dimensional mixtures $f=0.5 \ca{N}(\f{u},\f{I})+0.5 \ca{N}(-\f{u},\f{I})$ and $f' = 0.5 \ca{N}(4\f{u},\f{I})+0.5 \ca{N}(-2\f{u},\f{I})$. The mean of $f'$ is $\f{u}$. Applying Theorem \ref{thm:prev}, we get an upper bound of $\|f-f'\|_{\s{H}^2} =O(\|\f{u}\|_2+\|\f{u}\|_2^2+\|\f{u}\|_2^3)= O(\|\f{u}\|_2)$, when $\|\f{u}\|_2 \ll 1$. In contrast, Theorem \ref{thm:main_high_dim_tv} only gives a lower bound of $\Omega(\|\f{u}\|_2^2)$, which can be much smaller than $\|\f{u}\|_2$. It would be an interesting open direction to derive tight bounds on this instance. We do not know if this is an inherent limitation of either of the bounds, and it may be possible to extend our results to capture the $\|\f{u}\|_2$ term, or tighten the upper bound.

\subsection{Preliminaries}

We use $\Omega(\cdot)$, $O(\cdot)$, and $\Theta(\cdot)$ to hide absolute constants. For vectors $\f{u},\f{v} \in \mathbb{R}^d$, we let $\ip{\f{u}}{\f{v}}$ denote the Euclidean inner product.
We use the characteristic function of a distribution, defined below.
\begin{defn} The characteristic function $C_f: \mathbb{R} \to \mathbb{C}$ of a distribution $f$ is 
$
C_f(t) = \int_{\mathbb{R}} e^{itx} f(x) dx.
$
\end{defn}
\noindent
If $X$ is a random variable with distribution $f$, then $C_f(t) = \bb{E}_{X \sim f}[ e^{itX}]$. The characteristic function of a two-component, one-dimensional mixture $f_{\mu_0,\mu_1}\in \ca{F}$ is
% , it is known that 
$
    C_{f_{\mu_0,\mu_1}}(t) =\frac12 e^{-\sigma^2t^2/2}(e^{it\mu_0}+e^{it\mu_1}).
$
The characteristic function can be used to bound the TV distance with the following lemma. 
\begin{lemma}[\cite{krishnamurthy20a}]\label{lem:chartv}
For distributions $f,f'$ on a shared sample space $\Omega\subseteq\mathbb{R}$, 
$$
\left\|f -f'\right\|_{\s{TV}} \ge \frac{1}{2} \sup_{t \in \bb{R}}|C_f(t) -C_{f'}(t)|.
$$
\end{lemma}

\paragraph{Organization.} The rest of the paper is organized as follows. In 
Section~\ref{sec:tech}, we give a brief high level overview of our results. In Section~\ref{app:complex}, we provide the main parts of the proof of our TV distance result for one dimensional mixtures, which is based on elementary complex analysis. Subsequently, in Section~\ref{sec: highdim-tv-lb}, we provide the proof of the TV distance lower bound in the high dimensional case.

\section{Technical Overview}\label{sec:tech}
In one dimension, we lower bound the TV distance as follows. 
For $f_{\mu_0,\mu_1},f_{\mu_0',\mu_1'} \in \ca{F}$, suppose
$\mu_0$ is the smallest mean. Recall that $\delta_1 = \max\{|\mu_0-\mu_1|, |\mu_0'-\mu_1'|\} $ and
$\delta_2 = \max\{|\mu_0'-\mu_0|, |\mu_1-\mu_1'|\} $.
If $[\mu_0',\mu_1'] \subseteq [\mu_0,\mu_1]$,
then 
$
    ||f_{\mu_0,\mu_1}-f_{\mu_0',\mu_1'}||_{\s{TV}} \ge \Omega(\min(1, \delta_1\delta_2 /\sigma^2))
$
and otherwise,
$
    ||f_{\mu_0,\mu_1}-f_{\mu_0',\mu_1'}||_{\s{TV}} \ge \Omega(\min(1, \delta_2 /\sigma)).
$
The latter case corresponds to when either both means from one mixture are smaller than another, i.e., $\mu_0 \leq  \mu_1 \leq \mu_0', \mu_1'$, or the mixtures' means are interlaced, i.e., $\mu_0 \leq \mu_0' \leq \mu_1\leq \mu_1'$.

We use Lemma \ref{lem:chartv} to lower bound the TV distance between mixtures $f_{\mu_0,\mu_1}, f_{\mu_0',\mu_1'} \in \ca{F}$ by the modulus of a complex analytic function:
\begin{equation}
\label{eq: lowerbound_tv-prelim}
4 \left |\left |f_{\mu_0,\mu_1} -f_{\mu_0',\mu_1'}\right|\right|_{\s{TV}} 
 \ge 
\sup_t  e^{-\frac{\sigma^2 t^2}{2}} \left |e^{i t \mu_0} + e^{i t \mu_1} - e^{i t \mu'_0}-e^{i t \mu'_1} \right |.
\end{equation}
Let $h(t)=  e^{i t \mu_0} + e^{i t \mu_1} - e^{i t \mu'_0}-e^{i t \mu'_1}$.
A lower bound on $||f_{\mu_0,\mu_1} -f_{\mu_0',\mu_1'}||_{\s{TV}} $ can be obtained by 
taking $t=1/(c\sigma)$ for $c$ constant, so that $e^{- \sigma^2 t^2/2}$ is not too small. Then, it remains to bound $ |h(t) |$ at the chosen value of $t$. In some cases, we will have to choose the constant $c$ very carefully, as terms in $h(t)$ can cancel out due to the periodicity of the complex exponential function. For instance, if $\mu_0=0$, $\mu_1=200 \sigma$, $\mu_0' = \sigma$, and $\mu_1'=201\sigma$ with $\sigma = 2 \pi$, then $|h(1)|=0$.

It is reasonable to wonder whether there is a simple, global way to lower bound Eq.~(\ref{eq: lowerbound_tv-prelim}).
We could reparameterize the function  $h(t)$ as the complex function $g(z) = z^{\mu_0} +  z^{\mu_1}-z^{\mu_0'}-z^{ \mu_1'}$, where $z = e^{it}$,
then study $|g(z)|$, for $z$ in the disc with center 0 and radius 1 in the complex plane.
However, we are unaware of a global way to bound $|g(z)|$ here due to the fact that (i) $g(z)$ is not analytic at 0 when the means are non-integral and (ii) there is not a clear, large lower bound for $g(z)$ anywhere inside the unit disc.
These two facts obstruct the use of either the Maximum Modulus Principle or tools from harmonic measure to obtain lower bounds. Instead, we use a series of lemmas to handle the different ways that $|h(t)|$ can behave. The techniques include basic complex analysis and Taylor series approximations of order at most three.

% \subsubsection*{High-Dimensional Overview}

% Here, we prove a lower bound on the total variation distance between any two distinct candidate distributions $f_{\f{\mu}_0,\f{\mu}_1}, f_{\f{\mu}_0',\f{\mu}_1'} \in \ca{F}$.
Let $f_{\f{\mu}_0,\f{\mu}_1}^{\f{t}}$ be the distribution of the samples obtained according to $f_{\f{\mu}_0,\f{\mu}_1}$ and projected onto the direction $\f{t} \in \bb{R}^d$.
We have (see Lemma \ref{lem:simple} for a proof)
\begin{align*}
    f_{\f{\mu}_0,\f{\mu}_1}^{\f{t}} \equiv \frac{\ca{N}(\f{\mu}_0^{T}\f{t},\f{t}^T\f{\Sigma}\f{t})}{2}+\frac{\ca{N}(\f{\mu}_1^{T}\f{t},\f{t}^{T}\f{\Sigma}\f{t})}{2}  \text{,} \quad f_{\f{\mu}_0',\f{\mu}_1'}^{\f{t}} \equiv \frac{\ca{N}(\f{\mu}_0'^{T}\f{t},\f{t}^{T}\f{\Sigma}\f{t})}{2}+\frac{\ca{N}(\f{\mu}_1'^{T}\f{t},\f{t}^{T}\f{\Sigma}\f{t})}{2}.
\end{align*}
By the data processing inequality for $f$-divergences (see Theorem 5.2 in \cite{devroye2012combinatorial}), we have
$\| f_{\f{\mu}_0,\f{\mu}_1}-f_{\f{\mu}_0',\f{\mu}_1'} \|_{\s{TV}} \ge \sup_{\f{t} \in \bb{R}^d} \| f_{\f{\mu}_0,\f{\mu}_1}^{\f{t}} -f_{\f{\mu}_0',\f{\mu}_1'}^{\f{t}} \|_{\s{TV}}.$
Using our lower bound on the TV distance between one-dimensional mixtures (Theorem \ref{thm:main1}), we obtain a lower bound on $\|f_{\f{\mu}_0,\f{\mu}_1}-f_{\f{\mu}_0',\f{\mu}_1'} \|_{\s{TV}}$ by choosing $\f{t} \in \bb{R}^d$ carefully. This  leads to Theorem~\ref{thm:main_high_dim_tv}.

%%%%%%%%%%%%%%%%%%
\section{Lower Bound on TV Distance of 1-Dimensional Mixtures}\label{app:complex}

Consider distinct Gaussian mixtures $f_{\mu_0,\mu_1}, f_{\mu_0',\mu_1'} \in \ca{F}$.
% \begin{align*}
%     f_{\mu_0,\mu_1} \equiv \frac{\ca{N}(\mu_0,\sigma^2)}{2}+\frac{\ca{N}(\mu_1,\sigma^2)}{2} \quad \text{and} \quad f_{\mu_0',\mu_1'} \equiv \frac{\ca{N}(\mu_0',\sigma^2)}{2}+\frac{\ca{N}(\mu_1',\sigma^2)}{2}.
% \end{align*}
% such that $\sigma^2$ is known. 
Without loss of generality we will also let $\mu_0 \le \min(\mu_1,\mu_0',\mu_1')$ be the smallest unknown parameter, and let $\mu_1' \ge \mu_0'$. We maintain these assumptions throughout this section, and we will prove Theorem~\ref{thm:main1}. 

% By using Lemma \ref{lem:chartv},  we can lower bound the TV distance between $f_{\mu_0,\mu_1}$ and $f_{\mu_0',\mu_1'}$ with the modulus of a complex analytic function:
% \begin{equation}\label{eq: lowerbound_tv}
% 2\left|\left|f_{\mu_0,\mu_1} -f_{\mu_0',\mu_1'}\right|\right|_{\s{TV}}
% \ge \sup_t  \frac{e^{-\frac{\sigma^2 t^2}{2}}}{2} \left |e^{i t \mu_0} + e^{i t \mu_1} - e^{i t \mu'_0}-e^{i t \mu'_1} \right |.
% \end{equation}
% {\color{red} AM: Eq. 2 and 3 are same - this is redundant.}

Eq.~(\ref{eq: lowerbound_tv-prelim}) implies that we can lower bound $||f_{\mu_0,\mu_1} -f_{\mu_0',\mu_1'}||_{\s{TV}}$ by the modulus of a complex analytic function with parameter $t$. Then, we can optimize the bound by choosing $t = \Theta(1/\sigma)$ and lower bounding the term in the absolute value signs.

We define the following parameters relative to the means to simplify some bounds:
\begin{align*}
\delta_1 &= \max(|\mu_0-\mu_1|, |\mu_0'-\mu_1'|) \qquad \delta_2 = \max(|\mu_0'-\mu_0|, |\mu_1-\mu_1'|)\\
\delta_3 &= \left| \mu_0+\mu_1-\mu_0'-\mu_1' \right | \qquad \hspace{8.5mm}
\delta_4 = \min(\left| \mu_0'-\mu_0 \right|, \left| \mu_1'-\mu_1 \right|).
\end{align*}
% \begin{itemize}
%     \item $\delta_1 = \max\{|\mu_0-\mu_1|, |\mu_0'-\mu_1'|\} $
%       \item $ \delta_2 = \max(|\mu_0'-\mu_0|, |\mu_1-\mu_1'|) $
%       \item $ \delta_3 = \left| \mu_0+\mu_1-\mu_0'-\mu_1' \right| $
%       \item $ \delta_4 = \min(\left| \mu_0'-\mu_0 \right|, \left| \mu_1'-\mu_1 \right|)$ 
% \end{itemize} 
We first consider $t$ such that $t(\mu_1-\mu_0),t(\mu_1'-\mu_0),t(\mu_1'-\mu_0) \leq \frac{\pi}{4}$, which is covered in Lemma~\ref{lem: sepmeans2comp}.

\begin{lemma}\label{lem: sepmeans2comp}
 For $t>0$ with $t(\mu_1-\mu_0),t(\mu_1'-\mu_0),t(\mu_1'-\mu_0) \in [0,\frac{\pi}{4}]$, if $\mu_0',\mu_1' \in [\mu_0,\mu_1]$, then
\[
\left |e^{i t \mu_0} + e^{i t \mu_1} - e^{i t \mu'_0}-e^{i t \mu'_1} \right | \ge  \max\left (\frac{t^2(\delta_1-\delta_4)\delta_4}{2}, \frac{t\delta_3}{4\sqrt{2}}\right ) 
\]
and otherwise, when $\mu_1' > \mu_1$,
$
\left |e^{i t \mu_0} + e^{i t \mu_1} - e^{i t \mu'_0}-e^{i t \mu'_1} \right | \ge t\delta_2/ (2\sqrt{2}).
$
\end{lemma}
\begin{figure}
    \centering
    \includegraphics[width=15cm]{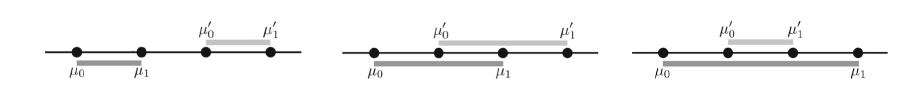}
    \caption{Layout of the means for Theorem~\ref{thm:main1}. The means can be ordered in different ways, which affects the analysis of lower bounding $|e^{it \mu_0}+e^{i t\mu_1}-e^{i t \mu_0'}-e^{it \mu_1'}|$ in Lemma~\ref{lem: sepmeans2comp}.
    For a fixed $t$, the order affects (i) whether the real or imaginary part of $e^{it \mu_0}+e^{i t\mu_1}-e^{i t \mu_0'}-e^{it \mu_1'}$ has large modulus and (ii) whether the terms from $\mu_0$ and $\mu_1$ or $\mu_0'$ and $\mu_1'$ dominate. }
    \label{fig:ordering-means}
\end{figure}
See Figure~\ref{fig:ordering-means} for an illustration of the different ways that the means can be ordered.
The lemma follows from straightforward calculations that only use Taylor series approximations, trigonometric  identities, and basic facts about complex numbers. We include the proof  in Appendix~\ref{sec:lemma-proofs}.

Recall that we will choose $t= \Theta(1/\sigma)$ to cancel the exponential term in Eq.~(\ref{eq: lowerbound_tv-prelim}). Therefore,  Lemma~\ref{lem: sepmeans2comp} handles the case when all the means are within some interval of size $\Theta(\sigma)$.

Next, we prove that when the separation between the mixtures is substantially fair apart---when either $|\mu_0-\mu_0'|$ or $|\mu_1-\mu_1'|$ is at least $2 \sigma$---we have a constant lower bound on the TV distance. Recall that it is without loss of generality to assume that $\mu_0$ is the smallest parameter and $\mu_1' > \mu_0'$. A similar result as the following two lemmas has been observed previously (e.g.,~\cite{hardt2015tight}) but we provide a simple and self-contained proof.

\begin{lemma}{\label{lem:large_gap}}
If $\max(\left|\mu_0-\mu_0' \right|, \left|\mu_1-\mu_1' \right|) \ge 2\sigma$, then 
%one sample is sufficient to distinguish between $f_{\mu_0,\mu_1}$ and $f_{\mu_0',\mu_1'}$ with probability at least $\frac{1}{2}\Big(1-\sqrt{\frac{2}{\pi e}}\Big)$. 
it follows that $||f_{\mu_0,\mu_1} - f_{\mu_0',\mu_1'}||_{\s{TV}} \geq \Omega(1)$.
\end{lemma}

\begin{proof}
Assume that $\left|\mu_0-\mu_0' \right| \ge 2\sigma$, where the case $\left|\mu_1-\mu_1' \right| \ge 2\sigma$ is analogous.
Recall from the definition of TV distance that 
\begin{align*}
    ||f_{\mu_0,\mu_1} - f_{\mu_0',\mu_1'}||_{\s{TV}} &\triangleq   \sup_{\ca{A}\subseteq \Omega}\Big(f_{\mu_0,\mu_1}(\ca{A})-f'_{\mu_0',\mu_1'}(\ca{A})\Big) \\
    &\ge \Pr_{X \sim f_{\mu_0,\mu_1}}[X \le \mu_0+\sigma] - \Pr_{Y \sim f_{\mu_0',\mu_1'}}[Y \le \mu_0+\sigma].
\end{align*}

For a random variable $X\sim f_{\mu_0,\mu_1}$,
let $\ca{E}$ denote the event that  we choose the component with mean $\mu_0$, i.e., if $\widetilde X$ denotes $X$ conditioned on $\ca{E}$, then we have $\widetilde X\sim \ca{N}(\mu_0,\sigma^2)$. Since the mixing weights are equal, we have $\Pr(\ca{E})=\Pr(\ca{E}^c)=1/2$, where $\ca{E}^c$ is the complement of $\ca{E}$. Therefore,
\begin{align}\label{gaussian-ub}
    \nonumber
    \Pr(X \ge \mu_0+\sigma ) &\le \Pr(\ca{E})\Pr(X \ge \mu_0+\sigma \mid \ca{E})+\Pr(\ca{E}^c)  = \frac{1}{2} \int_{\mu_0+\sigma}^{\infty} \frac{e^{-\frac{(t-\mu_0)^2}{2\sigma^2}}}{\sqrt{2\pi}\sigma} dt+\frac{1}{2}  \\
    &\le \frac{1}{2}\int_{\mu_0+\sigma}^{\infty} \Big(\frac{t-\mu_0}{\sigma}\Big) \frac{e^{-\frac{(t-\mu_0)^2}{2\sigma^2}}}{\sqrt{2\pi}\sigma} dt+\frac{1}{2} \le  \frac{1}{2}\cdot\frac{e^{-\frac12}}{\sqrt{2\pi}}+\frac{1}{2}.
\end{align}
Recall that $\mu_0 \leq \mu'_0$, $\mu_0' \le \mu_1'$ and $\left|\mu_0-\mu_0' \right| \ge 2\sigma$.
Again, for a random variable $Y\sim f_{\mu_0',\mu_1'}$, let $\ca{E}'$ denote the event that the component with mean $\mu_0'$ is chosen (and $\ca{E}'^c$ denotes $\mu_1'$ is chosen). 
% i.e. conditioned on $\ca{E}'$, we must have $Y\sim \ca{N}(\mu_0',\sigma^2)$; note that $\Pr(\ca{E}')=1/2$. 
Then,
\begin{align*}
    \Pr(Y \le \mu_0+\sigma) &= \Pr(\ca{E}')\Pr(Y \le \mu_0+\sigma\mid \ca{E}')+\Pr(\ca{E}'^{c})\Pr(Y \le \mu_0+\sigma\mid \ca{E}'^{c}) \\
    &\stackrel{a}{=} \Pr(\ca{E}')\Pr(Y \le \mu_0'-\sigma\mid \ca{E}')+\Pr(\ca{E}'^{c})\Pr(Y \le \mu_1'-\sigma\mid \ca{E}'^{c}) \\
    &\stackrel{b}{=} \Pr(\ca{E}')\Pr(Y \ge \mu_0'+\sigma\mid \ca{E}')+\Pr(\ca{E}'^{c})\Pr(Y \ge \mu_1'+\sigma\mid \ca{E}'^{c}) \\
    &\stackrel{c}{\le}    \frac{1}{2}\cdot\frac{e^{-\frac12}}{\sqrt{2\pi}}+\frac{1}{2}\cdot\frac{e^{-\frac12}}{\sqrt{2\pi}} = \frac{e^{-\frac12}}{\sqrt{2\pi}}
\end{align*}
where in step (a), we used the fact that $\mu_0'-\sigma \ge \mu_0+\sigma$ and $\mu_1'-\sigma \ge \mu_0+\sigma$; in step (b), we used the symmetry of Gaussian distributions; in step (c), we used the same analysis as in \eqref{gaussian-ub}. By plugging this in the definition of TV distance, we have
%It follows that with probability $\frac{1}{2}\Big(1-\sqrt{\frac{2}{\pi e}}\Big)$ the sample will be in the interval  $[\mu_0-\sigma,\mu_0+\sigma]$. 
% It follows from the definition of total variation distance that 
\begin{align*}
    ||f_{\mu_0,\mu_1} - f_{\mu_0',\mu_1'}||_{\s{TV}} \ge \Pr_{X \sim f_{\mu_0,\mu_1}}[X \le \mu_0+\sigma] - \Pr_{Y \sim f_{\mu_0',\mu_1'}}[Y \le \mu_0+\sigma] \ge \frac{1}{2}-\sqrt{\frac{9}{8\pi e}} \ge 0.137.
\end{align*}
\end{proof}

If Lemma~\ref{lem:large_gap} does not apply, then we case on whether $\max(\left|\mu_0-\mu_1\right|,\left|\mu_0'-\mu_1'\right|)$ is large or not. If $\max(\left|\mu_0-\mu_1\right|,\left|\mu_0'-\mu_1'\right|) < 100 \sigma$, we use Lemma~\ref{lem: sepmeans2comp}---exactly how will be explained later---and otherwise we use the following lemma. Recall that $\delta_2 = \max(|\mu_0'-\mu_0|, |\mu_1-\mu_1'|) $.

\begin{lemma}{\label{lem:small_prec}}
If $\max(\left|\mu_0-\mu_1\right|,\left|\mu_0'-\mu_1'\right|) \ge 100 \sigma$ and $\max(\left|\mu_0-\mu_0' \right|, \left|\mu_1-\mu_1' \right|) \le 2\sigma$,  then 
\begin{align*}
\sup_t e^{-\frac{\sigma^2 t^2}{2}}\left |e^{i t \mu_0} + e^{i t \mu_1} - e^{i t \mu'_0}-e^{i t \mu'_1} \right | \ge  \frac{\pi^2 \delta_2}{240e\sigma}.
\end{align*}
\end{lemma}
We defer the proof of Lemma~\ref{lem:small_prec} to Appendix~\ref{sec:lemma-proofs}.
Using Lemmas \ref{lem: sepmeans2comp}, \ref{lem:large_gap}, and \ref{lem:small_prec},
we prove Theorem \ref{thm:main1}.

\begin{proof}[Proof of Theorem \ref{thm:main1}]
Using Lemma \ref{lem:chartv}, we see that
\begin{align*}
2\left|\left|f_{\mu_0,\mu_1} -f_{\mu_0',\mu_1'}\right|\right|_{\s{TV}}  \ge \sup_t  \frac{e^{-\frac{\sigma^2 t^2}{2}}}{2} \left |e^{i t \mu_0} + e^{i t \mu_1} - e^{i t \mu'_0}-e^{i t \mu'_1} \right |.
\end{align*}
\paragraph{Case 1:} 
Consider the case when $\mu_0,\mu_1,\mu_0',\mu_1'$ are in an interval of size at most $100\sigma$, i.e., 
\begin{align} \label{eq:case1}
    \max\Big(\left|\mu_1'-\mu_0\right|, \left|\mu_1-\mu_0\right|, \left|\mu_0'-\mu_0\right| \Big) \le 100\sigma. 
\end{align}

Recall $\delta_1 = \max\{|\mu_0-\mu_1|, |\mu_0'-\mu_1'|\} $, 
$\delta_2 = \max(|\mu_0'-\mu_0|, |\mu_1-\mu_1'|) $,
$ \delta_3 = \left| \mu_0+\mu_1-\mu_0'-\mu_1' \right| $,
$ \delta_4 = \min(\left| \mu_0'-\mu_0 \right|, \left| \mu_1'-\mu_1 \right|)$.
For $t=\pi/400\sigma$,
$
0 \le t\max\Big(\left|\mu_1'-\mu_0\right|, \left|\mu_1-\mu_0\right|, \left|\mu_0'-\mu_0\right| \Big) \le \frac{\pi}{4}. 
$
We have assumed that
$\mu_0\le \min(\mu_1,\mu_0',\mu_1')$ and $\mu_0'\leq\mu_1'$. This implies that $\mu_0 \leq \mu_0' \leq \mu_1' \leq \mu_1$ in the subcase when $\mu_0',\mu_1' \in [\mu_0,\mu_1]$. This also implies that $\delta_1= \left|\mu_1-\mu_0\right| \ge 2\delta_4$, a fact we will use later.
The inequality in Eq.~(\ref{eq:case1}) implies that the above value of $t=\pi/400\sigma$ satisfies the conditions of Lemma~\ref{lem: sepmeans2comp}. Then, when $\mu_0',\mu_1' \in [\mu_0,\mu_1]$, the first part of the lemma implies that
\begin{align*}
2\left|\left|f_{\mu_0,\mu_1} -f_{\mu_0',\mu_1'}\right|\right|_{\s{TV}}  \ge \max\Big( \frac{\pi^2(\delta_1-\delta_4)\delta_4}{640000e\sigma^2}, \frac{\pi\delta_3}{3200 \sqrt{2}e \sigma }\Big).
\end{align*}
% When $\mu_0',\mu_1' \in [\mu_0,\mu_1]$, our assumption on the means implies that 
Now we observe that  $\delta_3 \ge \delta_2-\delta_4$.
To see this, assume without loss of generality that $\delta_2=|\mu_0'-\mu_0|$ and $\delta_4 = \left|\mu_1'-\mu_1\right|$.
By the triangle inequality, we have that $\delta_3=\left| \mu_0+\mu_1-\mu_0'-\mu_1' \right| \ge |\mu_0'-\mu_0|-\left|\mu_1'-\mu_1\right|=\delta_2-\delta_4$.
We split up the calculations based on the value of $\delta_3$. 
If $\delta_3 \ge \frac{\delta_2}{2}$, then
$
||f_{\mu_0,\mu_1} -f_{\mu_0',\mu_1'}||_{\s{TV}}  \ge \pi\delta_2/ (12800 \sqrt{2}e \sigma ).
$
On the other hand, if $\delta_3 \le \frac{\delta_2}{2}$, then since $\delta_3 \ge \delta_2 - \delta_4$, we have that $\delta_4 \ge \frac{\delta_2}{2}$. Coupled with the fact that $\delta_1 \geq 2\delta_4$ (hence $\delta_4 \le \delta_1/2$ implying $\delta_1 -\delta_4 \ge \delta_1/2$), we have that 
$
    ||f_{\mu_0,\mu_1} -f_{\mu_0',\mu_1'}||_{\s{TV}} \ge \pi^2\delta_1\delta_2/(5120000e \sigma^2).
$
Putting these together, we have 
\begin{align*}
    \left|\left|f_{\mu_0,\mu_1} -f_{\mu_0',\mu_1'}\right|\right|_{\s{TV}} \ge \min\Big(\frac{\pi^2\delta_1\delta_2}{5120000e \sigma^2},\frac{\pi\delta_2}{12800 \sqrt{2}e \sigma }\Big) = \frac{\pi^2\delta_1\delta_2}{5120000e \sigma^2} 
\end{align*}

For the case when both of $\mu_0',\mu_1'$ are not in $ [\mu_0,\mu_1]$, we have $\mu_1' > \mu_1$ (recall that $\mu_0$ is the smallest mean and $\mu_0'\leq\mu_1'$), and we can use the second part of Lemma~\ref{lem: sepmeans2comp} to conclude that
\begin{align*}
2\left|\left|f_{\mu_0,\mu_1} -f_{\mu_0',\mu_1'}\right|\right|_{\s{TV}} \ge \sup_t \frac{e^{-\frac{\sigma^2 t^2}{2}}}{2}\left |e^{i t \mu_0} + e^{i t \mu_1} - e^{i t \mu'_0}-e^{i t \mu'_1} \right | \ge \frac{\pi \delta_2}{1600\sqrt{2}e \sigma }.
\end{align*}

\paragraph{Case 2:} Next, consider when $\delta_2 = \max(|\mu_0'-\mu_0|, |\mu_1-\mu_1'|) \ge 2\sigma$. Lemma \ref{lem:large_gap} implies that 
\begin{align*}
    \left|\left|f_{\mu_0,\mu_1}-f_{\mu_0',\mu_1'}\right|\right|_{\s{TV}} \ge \Omega(1).
\end{align*}

\paragraph{Case 3:} Now, we consider the only remaining case, when $\delta_1=\max(\left|\mu_0-\mu_1\right|,\left|\mu_0'-\mu_1'\right|) \ge 100 \sigma$ and $\delta_2 \le \max(\left|\mu_0-\mu_0' \right|, \left|\mu_1-\mu_1' \right|) \le 2\sigma$. This case satisfies the conditions of Lemma \ref{lem:small_prec}, and therefore, we have that
\[
    2\left|\left|f_{\mu_0,\mu_1} -f_{\mu_0',\mu_1'}\right|\right|_{\s{TV}} \ge \sup_t \frac{e^{-\frac{\sigma^2 t^2}{2}}}{2}\left |e^{i t \mu_0} + e^{i t \mu_1} - e^{i t \mu'_0}-e^{i t \mu'_1} \right | \ge \frac{\pi^2 \delta_2}{240e\sigma},
\]
thus proving the theorem.
\end{proof}

\section{Lower Bound on TV Distance of \textit{d}-Dimensional Mixtures}
\label{sec: highdim-tv-lb}

We lower bound the TV distance of high-dimensional mixtures in $\mathcal{F}$ and prove  Theorem~\ref{thm:main_high_dim_tv}. 
% To simplify notation, we introduce some definitions. Let 
% \begin{align*}
%     &\delta_1=\max\{\left|\left|\f{\mu}_0-\f{\mu}_1\right|\right|_2,\left|\left|\f{\mu}_0'-\f{\mu}_1'\right|\right|_2\} \quad \delta_2=\max\{\left|\left|\f{\mu}_0'-\f{\mu}_0\right|\right|_2,\left|\left|\f{\mu}_1'-\f{\mu}_1\right|\right|_2\} \quad \\ &\text{and} \quad \delta_3=\max\{\left|\left|\f{\mu}_0'-\f{\mu}_1\right|\right|_2,\left|\left|\f{\mu}_1'-\f{\mu}_0\right|\right|_2\}
% \end{align*}
For any direction $\f{t} \in \bb{R}^d$,  we denote the projection of the distributions $f_{\f{\mu}_0,\f{\mu}_1}$ and $f_{\f{\mu}_0',\f{\mu}_1'}$ on~$\f{t}$ by $f_{\f{\mu}_0,\f{\mu}_1}^{\f{t}}$ and $f_{\f{\mu}_0',\f{\mu}_1'}^{\f{t}}$, respectively. 
The next lemma allows us to precisely define the projected mixtures.
\begin{lemma}\label{lem:simple}
For a random variable $\f{x} \sim \frac12 \ca{N} (\f{\mu}_0,\f{\Sigma})+\frac12 \ca{N} (\f{\mu}_1,\f{\Sigma})$, for any $\f{t}\in \bb{R}^d$, 
\begin{align*}
\f{t}^T\f{x}  \sim  \frac{\ca{N}(\langle\f{\mu}_0,\f{t}\rangle,\f{t}^T\f{\Sigma}\f{t})}{2}+\frac{\ca{N}(\ip{\f{\mu}_1}{\f{t}},\f{t}^{T}\f{\Sigma}\f{t})}{2}.   
\end{align*}
\end{lemma}

\begin{proof}
A linear transformation of a multivariate Gaussian is also a Gaussian. For $\f{x} \sim \ca{N}(\f{\mu}_0,\f{\Sigma})$, we see that $\ip{\f{t}}{\f{x}} \sim \ca{N}(\ip{\f{\mu}_0}{\f{t}},\f{t}^T\f{\Sigma}\f{t})$ by a computation of the mean and variance. Similarly, for $\f{x} \sim \ca{N}(\f{\mu}_1,\f{\Sigma})$, we have $\ip{\f{t}}{\f{x}} \sim \ca{N}(\ip{\f{\mu}_1}{\f{t}},\f{t}^T\f{\Sigma}\f{t})$. Putting these together, the claim follows.
\end{proof}

From Lemma \ref{lem:simple}, we can exactly define the one-dimensional mixtures
\begin{align*}
    f_{\f{\mu}_0,\f{\mu}_1}^{\f{t}} = \frac{\ca{N}(\ip{\f{\mu}_0}{\f{t}},\f{t}^T\f{\Sigma}\f{t})}{2}+\frac{\ca{N}(\ip{\f{\mu}_1}{\f{t}},\f{t}^{T}\f{\Sigma}\f{t})}{2} \text{, }  f_{\f{\mu}_0',\f{\mu}_1'}^{\f{t}} = \frac{\ca{N}(\ip{\f{\mu}_0'}{\f{t}},\f{t}^{T}\f{\Sigma}\f{t})}{2}+\frac{\ca{N}(\ip{\f{\mu}_1'}{\f{t}},\f{t}^{T}\f{\Sigma}\f{t})}{2}.
\end{align*}
% Notice that both $f_{\f{\mu}_0,\f{\mu}_1}^{\f{t}}$ and $f_{\f{\mu}_0',\f{\mu}_1'}^{\f{t}}$ are mixtures of one-dimensional two-component Gaussians.
By using the data processing inequality, or the fact that variational distance is non-increasing under all mappings (see, for instance, Theorem 5.2 in \cite{devroye2012combinatorial}), it follows that
\begin{align*}
    \left|\left| f_{\f{\mu}_0,\f{\mu}_1}-f_{\f{\mu}_0',\f{\mu}_1'} \right|\right|_{\s{TV}} \ge \sup_{\f{t} \in \bb{R}^d} \left|\left| f_{\f{\mu}_0,\f{\mu}_1}^{\f{t}} -f_{\f{\mu}_0',\f{\mu}_1'}^{\f{t}} \right|\right|_{\s{TV}}.
\end{align*}

Let $\ca{H}$ be the set of  permutations on $\{0,1\}$. The following lemma has two cases based on whether the interval defined by one pair of mean's projections is contained in the interval defined by the other pair's projections.
\begin{lemma}\label{lem:high_dim1}

Let $\f{t}\in \bb{R}^d$ be any vector. 
 If $\sqrt{\f{t}^T \f{\Sigma} \f{t}} = \Omega \Big(\max \left(|\ip{\f{t}}{\f{\mu}_0-\f{\mu}_1}|,|\ip{\f{t}}{\f{\mu}_0'-\f{\mu}_1'}| \right)\Big),$ and either $\ip{\f{\mu}_0'}{\f{t}},\ip{\f{\mu}_1'}{\f{t}} \in [\ip{\f{\mu}_0}{\f{t}},\ip{\f{\mu}_1}{\f{t}}]$ or $\ip{\f{\mu}_0}{\f{t}},\ip{\f{\mu}_1}{\f{t}} \in [\ip{\f{\mu}_0'}{\f{t}},\ip{\f{\mu}_1'}{\f{t}}]$, then $||f_{\f{\mu}_0,\f{\mu}_1}^{\f{t}}-f_{\f{\mu}_0',\f{\mu}_1'}^{\f{t}}||_{\s{TV}} $ is at least
    \begin{align*}
        \Omega\Big( \hspace{-1mm}\min\Big(1, \frac{1}{\f{t}^T \f{\Sigma} \f{t}} \cdot 
         \max \hspace{-1mm} \left(|\ip{\f{t}}{\f{\mu}_0-\f{\mu}_1}|,|\ip{\f{t}}{\f{\mu}_0'-\f{\mu}_1'}| \right) \min_{\sigma \in \ca{H}}\max \hspace{-1mm} \left(|\ip{\f{t}}{\f{\mu}_0-\f{\mu}_{\sigma(0)}'}|,|\ip{\f{t}}{\f{\mu}_1-\f{\mu}_{\sigma(1)}'}|   \right)
        \Big)\Big).
    \end{align*}
 Otherwise, we have that
% If either $\f{\mu_0'}^{T}\f{t}\not \in [\f{\mu}_0^{T}\f{t},\f{\mu}_1^{T}\f{t}]$ or  $\f{\mu_1'}^{T}\f{t} \not \in [\f{\mu}_0^{T}\f{t},\f{\mu}_1^{T}\f{t}]$, then
$||f_{\f{\mu}_0,\f{\mu}_1}^{\f{t}}-f_{\f{\mu}_0',\f{\mu}_1'}^{\f{t}}||_{\s{TV}} $ is at least
\begin{align*}
    % \left|\left|f_{\f{\mu}_0,\f{\mu}_1}^{\f{t}}-f_{\f{\mu}_0',\f{\mu}_1'}^{\f{t}}\right|\right|_{\s{TV}} \ge 
    \Omega\Big(\min\Big(1, \frac{1}{\sqrt{\f{t}^T \f{\Sigma} \f{t}}} \cdot \min_{\sigma \in \ca{H}} 
     \max(|\ip{\f{t}}{\f{\mu}_0-\f{\mu}_{\sigma(0)}'}|,| \ip{\f{t}}{\f{\mu}_1-\f{\mu}_{\sigma(1)} '}|)
    \Big)\Big).
\end{align*}

\end{lemma}

\begin{proof}
The proof follows directly from Theorem \ref{thm:main1}. Note that in Theorem \ref{thm:main1}, we assumed the ordering of the means without loss of generality, i.e., $\mu_0\le \min(\mu_1,\mu_0',\mu_1')$ and $\mu_0'<\mu_1'$. However, taking a minimum over the set of permutations in $\ca{H}$ allows us to restate the theorem in its full generality.
\end{proof}

Now we are ready to provide the proof of Theorem \ref{thm:main_high_dim_tv}.

\subsection{Proof of Theorem \ref{thm:main_high_dim_tv}}

Let $$S_1= \{\f{\mu}_1-\f{\mu}_0,\f{\mu}_1'-\f{\mu}_0'\},   \quad  S_2 = \{\f{\mu}_0'-\f{\mu}_0,\f{\mu}_1'-\f{\mu}_1\}, \quad  S_3 = \{\f{\mu}_0'-\f{\mu}_1,\f{\mu}_1'-\f{\mu}_0\},$$
and
$$\f{v}_1 = \s{argmax}_{s \in S_1} || \f{s}||_2, \quad \f{v}_2 = \s{argmax}_{s \in S_2} || \f{s}||_2, \quad  \f{v}_3 = \s{argmax}_{s \in S_3} || \f{s}||_2.$$

We consider two cases below. Depending on the norm of $\f{v}_1$, we modify our choice of projection direction. In the first case, we do not have a guarantee on the ordering of the means, so we use the first part of Lemma~\ref{lem:high_dim1}. In the second case, we can use the better bound in the second part of the lemma after arguing about the arrangement of the means. 
%This leads to different orderings of the means after projection, and 
% , the two cases below correspond to the two cases in Theorem~\ref{thm:main1}.
% , depending on the ordering of the means after projecting to one dimension.

\paragraph{Case 1 ($2\left|\left|\f{v}_1\right|\right|_2 \ge \min(\left|\left|\f{v}_2\right|\right|_2,\left|\left|\f{v}_3\right|\right|_2)$ and $\sqrt{\lambda_{\Sigma,\ca{U}}}=\Omega(\left|\left|\f{v}_1\right|\right|_2)$):} 
We start with a lemma that shows the existence of a vector $\f{z}$ that is correlated with $\{\f{v}_1,\f{v}_2,\f{v}_3\}$. We use $\f{z}$ to define the direction $\f{t}$ to project the means on, while roughly preserving their pairwise distances.
\begin{lemma}\label{lem:exists}
For $\f{v}_1,\f{v}_2,\f{v}_3$ defined above,
there exists a vector $\f{z}\in \bb{R}^d$  such that $\left|\left|\f{z}\right|\right|_2 \le 10$, $\f{z}$ belongs to the subspace spanned by  $\f{v}_1,\f{v}_2,\f{v}_3$, 
 and
$
    |\langle \f{z},\f{v} \rangle| \ge \frac{||\f{v}||_2}{6}$ for all  $\f{v}\in \{\f{v}_1,\f{v}_2,\f{v}_3\}.$
\end{lemma}

\begin{proof} We use the probabilistic method.
 Let $\f{u}_1,\f{u}_2,\f{u}_3$ be orthonormal vectors forming a basis of the subspace spanned by $\f{v}_1,\f{v}_2$ and $\f{v}_3$; hence, we can write the vectors $\f{v}_1,\f{v}_2$ and $\f{v}_3$ as a linear combination of $\f{u}_1,\f{u}_2,\f{u}_3$. Let us define a vector $\f{z}$ randomly generated from the subspace spanned by $\f{v}_1,\f{v}_2,\f{v}_3$ as follows. Let $p,q,r$ be independently sampled according to $\ca{N}(0,1)$. Then, define
$
    \f{z} = p\f{u}_1+q\f{u}_2+r\f{u}_3.$
By this construction, we have that 
$
    \langle \f{z}, \f{v} \rangle \sim \ca{N}(0,\left|\left|\f{v}\right|\right|_2^2)$ for  all vectors  $\f{v}\in \{\f{v}_1,\f{v}_2,\f{v}_3\}, 
$
and further, $\left|\left|\f{z}\right|\right|_2^2 = p^2+q^2+r^2$. Hence, for any $\f{v}\in \{\f{v}_1,\f{v}_2,\f{v}_3\}$, we have 
\begin{align*}
    \Pr\left (|\langle \f{z},\f{v} \rangle| \le ||\f{v}||_2/6 \right ) &\le \int_{-\frac{||\f{v}||_2}{6}}^{\frac{||\f{v}||_2}{6}} \frac{e^{-x^2/2||\f{v}||_2^2}}{\sqrt{2\pi ||\f{v}||_2^2}}dx\\
    &\le \int_{-\frac{||\f{v}||_2}{6}}^{\frac{||\f{v}||_2}{6}} \frac{1}{\sqrt{2\pi ||\f{v}||_2^2}}dx \le \frac{||\f{v}||_2}{3\sqrt{2\pi ||\f{v}||_2^2}} \le  \frac{1}{3\sqrt{2\pi}}. 
\end{align*}
Also, we can bound the norm of $z$ by bounding $p,q,r$. We see that
\begin{align*}
\Pr(p>5) \le  \int_{5}^{\infty} \frac{e^{-x^2/2}}{\sqrt{2\pi}}dx  \le \frac{1}{5} \int_{5}^{\infty} \frac{xe^{-x^2/2}}{\sqrt{2\pi}}dx \le  \frac{e^{-12.5}}{5}.     
\end{align*}
Similarly, $\Pr(p<-5) \le e^{-12.5}/5$. Applying the same calculations to $q$ and $r$ and taking a union bound, we must have that with positive probability $\left|\left|\f{z}\right|\right|_2 \le \sqrt{p^2+q^2+r^2} \le \sqrt{75} \le 10$ and
$
    |\langle \f{z},\f{v} \rangle| \ge ||\f{v}||_2/6$  for  all $\f{v}\in \{\f{v}_1,\f{v}_2,\f{v}_3\},$
implying there exists a vector $\f{z}$ that satisfies the claim.
% in the lemma statement.
% This completes the proof of the lemma. 
\end{proof}
For this case, we will use the first part of Lemma \ref{lem:high_dim1}.
Let $\f{z}$ be the vector guaranteed by Lemma~\ref{lem:exists}. Setting $\f{t} = \frac{\f{z}}{\sqrt{\f{z}^{T}\f{\Sigma} \f{z}}}$,
% , and noting that $\f{t}^{T}\Sigma\f{t} =1$, 
then Lemma~\ref{lem:exists} implies that
\begin{align*}
    |\langle \f{t},\f{v} \rangle| = \frac{|\langle \f{z},\f{v} \rangle|}{\sqrt{\f{z}^{T}\f{\Sigma} \f{z}}}  \ge \frac{||\f{v}||_2}{6\sqrt{\f{z}^{T}\f{\Sigma} \f{z}}} \qquad  \text{ for  all } \f{v}\in \{\f{v}_1,\f{v}_2,\f{v}_3\}.
\end{align*}

Recall that we defined $\lambda_{\Sigma,\ca{U}} \triangleq \max_{\substack{\left|\left|\f{u}\right|\right|_2=1 \\ \f{u} \in \s{span}(\f{v}_1,\f{v}_2,\f{v}_3)}} \f{u}^{T}\Sigma\f{u}$ to be the maximum amount a unit norm vector $\f{u}$ belonging to the span of the vectors $\f{v}_1,\f{v}_2,\f{v}_3$ is stretched by the matrix $\f{\Sigma}$. Note that $\lambda_{\Sigma,\ca{U}}$ is also upper bounded by the maximum eigenvalue of $\f{\Sigma}$. Now,
using the fact that $\left|\left|\f{z}\right|\right|_2\leq 10$ and $\sqrt{\f{z}^{T}\f{\Sigma} \f{z}} \le \sqrt{\lambda_{\Sigma,\ca{U}}}\left|\left|\f{z}\right|\right|_2 \leq 10\sqrt{\lambda_{\Sigma,\ca{U}}}$, we obtain 
\begin{align}\label{eq:inner-lb}
    |\langle \f{t},\f{v} \rangle| \ge \frac{||\f{v}||_2}{60\sqrt{\lambda_{\Sigma,\ca{U}}}} \qquad \text{ for  all } \f{v}\in \{\f{v}_1,\f{v}_2,\f{v}_3\}.
    % \text{ and } \f{t}^{T}\Sigma\f{t} =1.
\end{align}
\noindent
The part of Lemma~\ref{lem:high_dim1} that we use depends on whether
$\sqrt{\f{t}^T \f{\Sigma} \f{t}} = \Omega \Big(\max \left(|\ip{\f{t}}{\f{v}_1}| \right)\Big)$ or not. However, the second part of the lemma is stronger and implies the first part. Therefore, we simply use the lower bound in the first part of the lemma, and we see that 
\begin{align*}
&\left|\left|f_{\f{\mu}_0,\f{\mu}_1}^{\f{t}}-f_{\f{\mu}_0',\f{\mu}_1'}^{\f{t}} \right|\right|_{\s{TV}} \\ 
&= \Omega\Big(\min\Big(1, 
    \max \left(|\ip{\f{t}}{\f{\mu}_0-\f{\mu}_1}|,|\ip{\f{t}}{\f{\mu}_0'-\f{\mu}_1'}| \right) \min_{\sigma \in \ca{H}}\max  \left(|\ip{\f{t}}{(\f{\mu}_0-\f{\mu}_{\sigma(0)}')}|,|\ip{\f{t}}{(\f{\mu}_1-\f{\mu}_{\sigma(1)}')}|   \right)
    \Big)\Big) \\
    &\stackrel{(a)}{=} \Omega\Big(\min\Big(1,
    \left|\ip{\f{t}}{\f{v}_1}\right| \min (\left|\ip{\f{t}}{\f{v}_2}\right|,\left|\ip{\f{t}}{\f{v}_3}\right|)
    \Big)\Big) 
    \stackrel{(b)}{=} \Omega\Big(\min\Big(1, \frac{\|\f{v}_1\|_2\min(\|\f{v}_2\|_2, \|\f{v}_3\|_2)}{\lambda_{\Sigma,\ca{U}}} \Big)\Big),
\end{align*}
wherein step (a), we used the following facts (from definitions):
\begin{align}
    \label{eq:fact1}
    &\max \left(|\ip{\f{t}}{\f{\mu}_0-\f{\mu}_1}|,|\ip{\f{t}}{\f{\mu}_0'-\f{\mu}_1'}| \right) \ge \left|\ip{\f{t}}{\f{v}_1}\right| \\ \label{eq:fact2}   
    & \max  \left(|\ip{\f{t}}{(\f{\mu}_0-\f{\mu}_{0}')}|,|\ip{\f{t}}{(\f{\mu}_1-\f{\mu}_{1}')}|   \right) \ge \left|\ip{\f{t}}{\f{v}_2}\right| \\ \label{eq:fact3}
    & \max  \left(|\ip{\f{t}}{(\f{\mu}_0-\f{\mu}_{1}')}|,|\ip{\f{t}}{(\f{\mu}_1-\f{\mu}_{0}')}|   \right) \ge \left|\ip{\f{t}}{\f{v}_3}\right|
\end{align}
and in step (b), we used Eq.~\eqref{eq:inner-lb} for each $\f{v}\in \{\f{v}_1,\f{v}_2,\f{v}_3\}$.

\paragraph{Case 2 ($2\left|\left|\f{v}_1\right|\right|_2 \le \min(\left|\left|\f{v}_2\right|\right|_2,\left|\left|\f{v}_3\right|\right|_2)$ or  $\sqrt{\lambda_{\Sigma,\ca{U}}}=O(\left|\left|\f{v}_1\right|\right|_2)$):}
For this case, we will use the second part of Lemma \ref{lem:high_dim1}.
The random choice of $\f{t}$ in Case 1 would have been sufficient for using the second part of Lemma \ref{lem:high_dim1} when $c\left|\left|\f{v}_1\right|\right|_2 \le \min(\left|\left|\f{v}_2\right|\right|_2,\left|\left|\f{v}_3\right|\right|_2)$ or  $\sqrt{\lambda_{\Sigma,\ca{U}}}=O(\left|\left|\f{v}_1\right|\right|_2)$ for some large constant $c$ but with a deterministic choice of $\f{t}$ that is described below, we can show that $c=2$ is sufficient.
 Let $\f{t} = \frac{\f{v}}{\sqrt{\f{v}^{T}\f{\Sigma} \f{v}}}$, where
 \begin{align*}
     \f{v}= \frac{\f{v}_2}{\left|\left|\f{v}_2\right|\right|_2}+\frac{s\f{v}_3}{\left|\left|\f{v}_3\right|\right|_2} \text{ with } s=\s{argmax}_{u \in\{-1,+1\}} \langle \f{v}_2,u\f{v}_3 \rangle.
 \end{align*}
Notice that we must have $s\langle \f{v}_2, \f{v}_3 \rangle >0$ from the definition of $s$.
% Recalling that $\delta_1 = \left|\left|\f{v}_1\right|\right|_2$ and $\delta_2 = \left|\left|\f{v}_2\right|\right|_2$ and $\delta_3 = \left|\left|\f{v}_3\right|\right|_2$, where the assumption 
Then we see that 
\begin{align}
    \nonumber 
    \left|\langle \f{v},\f{v}_1 \rangle\right| &= \left| \frac{\langle \f{v}_2, \f{v}_1 \rangle}{\left|\left|\f{v}_2\right|\right|_2}+\frac{s\langle \f{v}_3, \f{v}_1 \rangle}{\left|\left|\f{v}_3\right|\right|_2} \right| \le 2\left|\left|\f{v}_1\right|\right|_2 \le \min(\left|\left|\f{v}_2\right|\right|_2,\left|\left|\f{v}_3\right|\right|_3) \\
    \label{eq:fact5}
    \left|\langle\f{v},\f{v}_2\rangle\right| &= \left| \left|\left|\f{v}_2\right|\right|_2 +\frac{s\langle \f{v}_2,\f{v}_3 \rangle}{\left|\left|\f{v}_3\right|\right|_2} \right| \ge \left|\left|\f{v}_2\right|\right|_2  \\
    \label{eq:fact6}
    \left|\langle\f{v},\f{v}_3\rangle \right| &=  \left|\frac{\langle \f{v}_2,\f{v}_3 \rangle}{\left|\left|\f{v}_2\right|\right|_2} + s\left|\left|\f{v}_3\right|\right|_2 \right| = \left|\frac{s\langle \f{v}_2,\f{v}_3 \rangle}{\left|\left|\f{v}_2\right|\right|_2} + \left|\left|\f{v}_3\right|\right|_2 \right|  \ge \left|\left|\f{v}_3\right|\right|_2.  
\end{align}
The first inequality follows the norm bound on $\f{v}_1$ for this case, the second inequality uses that the definition of $\f{v}$ and $s$ imply that the second term in the sum is non-negative, and the third inequality uses the same logic and the fact that $s \in \{-1,1\}$.

We just showed that $\left|\langle\f{v},\f{v}_1\rangle\right| \leq \min(\left|\langle\f{v},\f{v}_2\rangle\right|, \left|\langle\f{v},\f{v}_3\rangle\right|)$, and hence 
% Using our choice of $\f{t}$,  
%  we must have 
 $\langle \f{t},\f{v}_1 \rangle \le \min(\langle \f{t},\f{v}_2  \rangle, \langle \f{t},\f{v}_3 \rangle)$.
% \begin{align*}
%     \f{t}^{T}\f{\Sigma} \f{t} =1 \quad {\text{and}} \quad \langle \f{t},\f{v}_1 \rangle \le \min(\langle \f{t},\f{v}_2  \rangle, \langle \f{t},\f{v}_3 \rangle),
% \end{align*}
This implies that the interval defined by one pair of projected means is not contained within the interval defined by the other pair of projected means.
% , and hence either one of $\f{\mu_0'}^{T}\f{t},\f{\mu_1'}^{\f{T}}\f{t} \not \in [\f{\mu}_0^{T}\f{t},\f{\mu}_1^{T}\f{t}]$ or one of $\f{\mu}_0^{T}\f{t},\f{\mu}_1^{\f{T}}\f{t} \not \in [\f{\mu_0'}^{T}\f{t},\f{\mu_1'}^{T}\f{t}]$.
This means we can use the second part of Lemma~\ref{lem:high_dim1}. Furthermore, we  also have $\f{t}^{T}\Sigma \f{t}=1$. Finally, since $\left|\left|\f{v}\right|\right|_2 \le 2$, note that $\sqrt{\f{v}^{T}\Sigma\f{v}} \le 2\sqrt{\lambda_{\Sigma,\ca{U}}}$. Using Lemma \ref{lem:high_dim1} with our choice of $\f{t}$, we see that
\begin{align*}
&\left|\left|f_{\f{\mu}_0,\f{\mu}_1}^{\f{t}}-f_{\f{\mu}_0',\f{\mu}_1'}^{\f{t}} \right|\right|_{\s{TV}} \\
&= \Omega\Big(\min\Big(1,\min_{\sigma \in \ca{H}} 
      \max(|\ip{\f{t}}{(\f{\mu}_0-\f{\mu}_{\sigma(0)}')}|,|\ip{\f{t}}{(\f{\mu}_1-\f{\mu}_{\sigma(1)} '})|)
    \Big)\Big) \\
&\stackrel{(a)}{=} \Omega\Big(\min\Big(1,\min\Big(\left|\ip{\f{t}}{\f{v}_2}\right|,\left|\ip{\f{t}}{\f{v}_3}\right|\Big)
    \Big)\Big)  \stackrel{(b)}{=} \Omega\Big(\min\Big(1, \frac{\min(\left|\left|\f{v}_2\right|\right|_2,\left|\left|\f{v}_3\right|\right|_2)}{\sqrt{\lambda_{\Sigma,\ca{U}}}} \Big) \Big).
\end{align*}
In step (a), we used Eq.~(\ref{eq:fact2}) and (\ref{eq:fact3}), while in step (b), we used Eq.~(\ref{eq:fact5}) and (\ref{eq:fact6}).
The remaining case is when $\sqrt{\lambda_{\Sigma,\ca{U}}}=O(\left|\left|\f{v}_1\right|\right|_2)$. The second part of Lemma \ref{lem:high_dim1} applies because we observe that $\sqrt{\f{t}^T \f{\Sigma} \f{t}} = O \Big(\max \left(|\ip{\f{t}}{\f{\mu}_0-\f{\mu}_1}|,|\ip{\f{t}}{\f{\mu}_0'-\f{\mu}_1'}| \right)\Big)$. To see this, recall that $\f{t}^T\f{\Sigma}\f{t}=1$, and hence,
\begin{align*}
    & \max \left(|\ip{\f{t}}{\f{\mu}_0-\f{\mu}_1}|,|\ip{\f{t}}{\f{\mu}_0'-\f{\mu}_1'}| \right)\\
    &\ge\left|\ip{\f{t}}{\f{v}_1}\right| = \left|\frac{\f{z}^{T}\f{v}_1}{\sqrt{\f{z}^{T}\f{\Sigma} \f{z}}}\right| \ge \frac{\left|\left|\f{v}_1\right|\right|_2}{6\sqrt{\f{z}^{T}\f{\Sigma} \f{z}}} \ge \frac{\left|\left|\f{v}_1\right|\right|_2}{6\sqrt{\lambda_{\Sigma,\ca{U}}}} =\Omega(1) = \Omega\left(\sqrt{\f{t}^T\f{\Sigma}\f{t}}\right).
\end{align*}

Next, recall that Lemma \ref{lem:exists} implies that $\left|\ip{\f{t}}{\f{v}_2}\right| \ge ||\f{v}_2||/6$ and $\left|\ip{\f{t}}{\f{v}_3}\right| \ge ||\f{v}_3||/6$. Then, using the second part of Lemma \ref{lem:high_dim1}, we have that
\begin{align*}
&\left|\left|f_{\f{\mu}_0,\f{\mu}_1}^{\f{t}}-f_{\f{\mu}_0',\f{\mu}_1'}^{\f{t}} \right|\right|_{\s{TV}} \\
&= \Omega\Big(\min\Big(1,\min_{\sigma \in \ca{H}} 
      \max(|\ip{\f{t}}{(\f{\mu}_0-\f{\mu}_{\sigma(0)}')}|,|\f{t}^{T}(\f{\mu}_1-\f{\mu}_{\sigma(1)} ')|)
    \Big)\Big) \\
&\stackrel{(a)}{=} \Omega\Big(\min\Big(1,\min\Big(\left|\ip{\f{t}}{\f{v}_2}\right|,\left|\ip{\f{t}}{\f{v}_3}\right|\Big)
    \Big)\Big)  \stackrel{(b)}{=} \Omega\Big(\min\Big(1, \frac{\min(\left|\left|\f{v}_2\right|\right|_2,\left|\left|\f{v}_3\right|\right|_2)}{\sqrt{\lambda_{\Sigma,\ca{U}}}} \Big) \Big).
\end{align*}
Again in step (a), we used Eq.~(\ref{eq:fact2}) and (\ref{eq:fact3}) while in step (b), we used Eq.~(\ref{eq:fact5}) and (\ref{eq:fact6}).  This completes the proof of Theorem \ref{thm:main_high_dim_tv}.

\section{Conclusion and Open Questions}

We demonstrated the use of complex analytic tools to prove new lower bounds on the total variation distance between any two Gaussian mixtures with two equally weighted components and shared component variance. For a pair of mixtures with shared component variance, we provide guarantees on the total variation distance as a function of the largest gap (among the two mixtures) between the component means. Although intuitive, such a characterization was missing despite a vast literature on the total variation distance between mixtures of Gaussians with two components. We also extended our results to high dimensions and showed an elegant way via characteristic functions to reduce the problem to the one-dimensional setting. Finally, we should also point out that our lower bounds hold for all pairs of Gaussian mixtures with shared component covariance matrix without any assumptions on the component means; this was not the case in the prior results, which either needed the means to be bounded or the means of both mixtures to be zero.

The complex analytic tools in this work are elementary, and there is room for development. These tools may be helpful in proving bounds on statistical distance between more diverse distributions. 
For example, our analytic techniques do extend to mixtures of two Gaussians with shared covariance and certain non-equal mixing weights. To give a specific instance, for a mixture with weights $c_0$ and $c_1$, we could replace Lemma \ref{lem: sepmeans2comp} so the lower bound only gains an additional multiplicative factor of $\min\{c_0,c_1\}$ when  $\mu_1' > \mu_1$.
We avoided stating our results in full generality of the mixing weights to not complicate our techniques and results.
% and of a term dependent on $c_0,c_1,\delta_4,\delta_1,$ and $t$ otherwise.}
 It would be useful and interesting to provide matching upper bounds on the total variation distance of two-component mixtures for all instances as a function of the means and covariance (generalizing the results for single Gaussians~\cite{devroye2018total}). Extending our results to more general mixtures with $k$ components and unknown component variances/weights (e.g., for Gaussian or even other families of distributions, such as those studied in~\cite{krishnamurthy20a}) will be of significant interest to both the statistics and machine learning communities.

\noindent \paragraph{Acknowledgement:} The work of A.~Mazumdar and S.~Pal is supported in part by NSF awards 2133484, 2127929, and 1934846.

\bibliographystyle{plain}
%\bibliography{references}

\appendix

\section{Missing Proofs from Section 3}\label{sec:lemma-proofs}
Here, we provide the proofs for Lemmas \ref{lem: sepmeans2comp} and \ref{lem:small_prec}. Recall that we have indexed the means such that $\mu_0\le \min(\mu_1,\mu_0',\mu_1')$ and $\mu_0'<\mu_1'$.

\begin{proof}[Proof of Lemma~\ref{lem: sepmeans2comp}]
We use case analysis on different orderings of the means and their separations.

\begin{claim}
For any $t>0$ such that $t(\mu_1-\mu_0),t(\mu_1'-\mu_0),t(\mu_1'-\mu_0) \in [0,\frac{\pi}{4}]$, when $\mu_1' > \mu_1$,
\begin{align*}
\left |e^{i t \mu_0} + e^{i t \mu_1} - e^{i t \mu'_0}-e^{i t \mu'_1} \right | \ge \frac{t\delta_2}{2\sqrt{2}}.
\end{align*}
\end{claim}

\begin{proof}
Assume that $\mu_1'-\mu_1\geq \mu_0'-\mu_0$, and recall that $\delta_2= \max(|\mu_0'-\mu_0|, |\mu_1-\mu_1'|)$. First, we factor out the lowest common exponent to see that 
$$\left|e^{it\mu_0}+e^{it\mu_1}-e^{it\mu_0'}-e^{it\mu_1'} \right| =\left|1+e^{it(\mu_1 - \mu_0)}-e^{it(\mu_0' - \mu_0)}-e^{it(\mu_1'- \mu_0)} \right|.$$
Let us denote $\phi_1=\mu_1-\mu_0$, $\phi_0'=\mu_0'-\mu_0$ and $\phi_1'=\mu_1'-\mu_0$. The following inequalities hold:
\begin{align*}
 \left|1+e^{it\phi_1}-e^{it\phi_0'}-e^{it\phi_1'} \right| 
& \ge  \left| \sin(t \phi_1)-\sin(t \phi_0')-\sin(t \phi_1') \right|  &&[ |z| \geq |\text{Im}(z)|]\\
& \ge  -\sin(t \phi_1)+\sin(t \phi_0')+\sin(t \phi_1' ) && [\text{Remove }|\cdot |]\\
% & \ge  -\sin(t \phi_1)+\sin(t \phi_1' ) && []\\
& \ge  -\sin(t \phi_1)+\sin(t (\phi_1 + (\phi_1' - \phi_1))) &&  [\sin(t \phi_0') \geq 0]\\
& = 2\sin\left (t \frac{\phi_1' - \phi_1}{2} \right)\cos\left (t\left (\phi_1+\frac{\phi_1' - \phi_1}{2}\right )\right)\\
& \ge  \frac{1}{2 \sqrt{2}} t \delta_2.
% && \cos(\cdot) \geq \sqrt{2}/2]\\ 
% & && \sin\left (t \frac{\phi_1' - \phi_1}{2} \right) \geq \frac{t \delta_2}{4}
\end{align*}
In the last line, we use that $\cos\left (t\left (\phi_1+\frac{\phi_1' - \phi_1}{2}\right )\right) \geq 1/\sqrt{2}$ and  $\sin\left (t \frac{\phi_1' - \phi_1}{2} \right) \geq \frac{t \delta_2}{4}$, where
the former follows from the fact that 
\begin{align*}
  0 \le  t\left (\phi_1+\frac{\phi_1' - \phi_1}{2}\right ) = \frac{t(\phi_1+ \phi_1')}{2} = \frac{1}{2}\Big(t(\mu_1-\mu_0)+t(\mu_1'-\mu_0)\Big) \le \frac{\pi}{4}
\end{align*}
and 
the latter follows from $\sin(x) \geq x/2$ for $x \in \mathbb{R}$.

If $\mu_0'-\mu_0 > \mu_1'-\mu_1$, then we can use a similar string of inequalities by using the fact that
$$\left|e^{it\mu_0}+e^{it\mu_1}-e^{it\mu_0'}-e^{it\mu_1'} \right| =\left|1+e^{it(\mu_0 - \mu_1)}-e^{it(\mu_0' - \mu_1)}-e^{it(\mu_1'- \mu_1)} \right|.$$
We denote $\phi_0=\mu_0-\mu_1$, $\phi_0'=\mu_0'-\mu_1$ and $\phi_1'=\mu_1'-\mu_1$. Note that all the $\phi$ are negative and $\phi_0 > \phi_0' > \phi_1'$. The following holds:
 
\begin{align*}
\left|e^{it\phi_0}+1-e^{it\phi_0'}-e^{it\phi_1'} \right| 
 & \geq  \left| \sin(t \phi_0)-\sin(t \phi_0')-\sin(t \phi_1') \right|  && \hspace{-7mm}[|z| \geq |\text{Re}(z)|]\\
  & =  \left| -\sin(t |\phi_0|)+\sin(t |\phi_0'|)+\sin(t |\phi_1'|) \right|  &&\hspace{-7mm}[\sin(\cdot) \text{ odd}]\\
  & =  -\sin(t |\phi_0|)+\sin(t |\phi_0'|)+\sin(t |\phi_1'|) && \hspace{-7mm} [\text{Remove }|\cdot|]    \\
  & \ge  -\sin(t |\phi_0|)+\sin(t (|\phi_0| + (|\phi_0'|- |\phi_0|))) &&  \hspace{-7mm}[\sin(t |\phi_0'|) \geq 0]\\
& = 2\sin\left(\frac{t (|\phi_0'| - |\phi_0|)}{2}\right)
\cos\left(t \left (|\phi_0|+\frac{|\phi_0'| - |\phi_0| }{2}\right )\right)\\
& \ge  \frac{1}{2 \sqrt{2}} t \delta_2.
% &&  [\cos(\cdot) \geq \sqrt{2}/2;\\ 
% & && \sin\left(\frac{t (|\phi_0'| - |\phi_0|)}{2}\right) \geq \frac{t \delta_2}{4} ]
\end{align*}
In the last line, we use that $\sin\left(\frac{t (|\phi_0'| - |\phi_0|)}{2}\right) \geq \frac{t \delta_2}{4}$ and
$\cos\left(t \left (|\phi_0|+\frac{|\phi_0'| - |\phi_0| }{2}\right )\right)\geq \frac{\sqrt{2}}{2}$.
\end{proof}

\begin{claim}\label{claim: other_factor_claim}
For $t>0$ such that $t(\mu_1-\mu_0),t(\mu_1'-\mu_0),t(\mu_1'-\mu_0) \in [0,\frac{\pi}{4}]$, if both $\mu_0',\mu_1' \in [\mu_0,\mu_1]$, then 
\begin{align*}
\left |e^{i t \mu_0} + e^{i t \mu_1} - e^{i t \mu'_0}-e^{i t \mu'_1} \right | \ge \max\left ( \frac{t^2(\delta_1-\delta_4)\delta_4}{2}, \frac{t\delta_3}{4\sqrt{2}}\right).
\end{align*}
\end{claim}

\begin{proof}

% The proof for the case when $\mu_0'-\mu_0 \geq \mu_1-\mu_1'$ is included in the main body; 
First, we show the left hand side of the inequality in the claim statement is at least $t \delta_3/ (4 \sqrt{2})$.

Assume that $\mu_0'-\mu_0=\delta_2$, recalling that $\delta_2= \max(|\mu_0'-\mu_0|, |\mu_1-\mu_1'|)$.
We factor out the lowest common exponent to see that 
$$\left|e^{it\mu_0}+e^{it\mu_1}-e^{it\mu_0'}-e^{it\mu_1'} \right| =\left|1+e^{it(\mu_1 - \mu_0)}-e^{it(\mu_0' - \mu_0)}-e^{it(\mu_1'- \mu_0)} \right|.$$
 Let us denote $\phi_1=\mu_1-\mu_0$, $\phi_0'=\mu_0'-\mu_0$ and $\phi_1'=\mu_1'-\mu_0$. 
 To prove following inequalities, we need two facts. We use {\bf Fact I} that $\frac{\partial}{\partial x} (\sin(x-y) - \sin(x)) = \cos(x-y) - \cos(x)\geq0$ for $\frac{\pi}{4} \ge x \ge y \ge 0$. In particular, taking $x=\phi_1$ and $y=\mu_1-\mu_1'$, the inequality is increasing with respect to $\phi_1$, so so we can lower bound the function at $\phi_1 = \mu_1-\mu_1'+\phi_0'$. Additionally, we use {\bf Fact II} that $-\sin(x+y) + 2\sin(x) \geq \sin((x-y)/2)\cos(y/2)$ for $0 \leq y \leq x \leq \pi/4$, for the choice of $x=\phi_0'$ and $y=\mu_1-\mu_1'$. Then, we have that
\begin{align*}
\left|1+e^{it\phi_1}-e^{it\phi_0'}-e^{it\phi_1'} \right| 
& \ge  \left| \sin(t \phi_1)-\sin(t \phi_0')-\sin(t \phi_1') \right|  &&[|z| \geq |\text{Im}(z)|]\\
& \ge  -\sin(t \phi_1)+\sin(t \phi_0')+\sin(t \phi_1' ) && [\text{Remove }|\cdot|]\\
& \ge  -\sin(t \phi_1)+\sin(t \phi_0')+\sin(t (\phi_1- (\mu_1-\mu_1')) ) \\
& \ge -\sin(t(\mu_1-\mu_1'+\phi_0')) + 2\sin(t\phi_0') && [\text{{\bf Fact I} above}] \\
 &\ge \sin \Big(\frac{t(-\mu_1+\mu_1'+\mu_0'-\mu_0)}{2}\Big)\cos\Big(\frac{t(\mu_1-\mu_1')}{2}\Big)  && [\text{{\bf Fact II} above}]\\
 & \ge \frac{t  \delta_3}{4 \sqrt{2}}.  && [\sin(x) \geq x/2; \\
 & && \cos(\cdot)\geq \sqrt{2}/2]
\end{align*}

Now, we assume that  $\mu_0'-\mu_0 < \mu_1-\mu_1'$ and factor out $e^{i t\mu_1}$:

$$\left|e^{it\mu_0}+e^{it\mu_1}-e^{it\mu_0'}-e^{it\mu_1'} \right| =\left|e^{it(\mu_0 - \mu_1)}+1-e^{it(\mu_0' - \mu_1)}-e^{it(\mu_1'- \mu_1)} \right|.$$
As in the proofs of other claims, we let $\phi_0=\mu_0-\mu_1$, $\phi_0'=\mu_0'-\mu_1$ and $\phi_1'=\mu_1'-\mu_1$. Then, we have

\begin{align*}
\left|e^{it\phi_0}+1-e^{it\phi_0'}-e^{it\phi_1'} \right| 
& \ge  \left| \sin(t \phi_0)-\sin(t \phi_0')-\sin(t \phi_1') \right|  &&[|z| \geq |\text{Im}(z)|]\\
& \ge  \left| -\sin(t |\phi_0|)+\sin(t |\phi_0'|)+\sin(t |\phi_1'|) \right|  &&[\sin(\cdot) \text{ odd}]\\
& \ge  -\sin(t |\phi_0|)+\sin(t |\phi_0'|)+\sin(t |\phi_1'| ) && [\text{Remove }|\cdot|]\\
& \ge  -\sin(t |\phi_0|)+\sin(t |\phi_0| - |\mu_0-\mu_0'|) +\sin(t |\phi_1'| ) \\
& \ge -\sin(t(|\phi_1'| + |\mu_0 - \mu_0'|)) + 2\sin(t|\phi_1'|) && [\text{{\bf Fact I} above}] \\
 &\ge \sin \Big(\frac{t(|\phi_1'| - |\mu_0 - \mu_0'|)}{2}\Big)\cos\Big(\frac{t|\mu_0-\mu_0'|}{2}\Big)  && [\text{{\bf Fact II} above}]\\
 & \ge \frac{t}{4 \sqrt{2}} \cdot  (|\phi_1'| - |\mu_0 - \mu_0'|)  = \frac{t}{4 \sqrt{2}} \cdot \delta_3. && [\sin(x) \geq x/2; \\
 & && \cos(\cdot)\geq \sqrt{2}/2]
\end{align*}
In the application of { \bf Fact I},  we let $|\phi_0|$ be as small as possible, choosing $|\phi_0| = |\phi_1'| + |\mu_0 - \mu_0'|$.

% \begin{claim}
% For $t>0$ such that $t(\mu_1-\mu_0),t(\mu_1'-\mu_0),t(\mu_1'-\mu_0) \in [0,\frac{\pi}{4}]$, when both $\mu_0',\mu_1' \in [\mu_0,\mu_1]$, then 
% \[
% \left |e^{i t \mu_0} + e^{i t \mu_1} - e^{i t \mu'_0}-e^{i t \mu'_1} \right | \ge \frac{t^2(\delta_1-\delta_4)\delta_4}{2}.
% \]
% \end{claim}
% \begin{proof}
Next, we show the left hand side of the inequality in the claim statement is at least $t^2 (\delta_1-\delta_4)\delta_4/ 2$.
Assume that $\mu_0'-\mu_0 \le \mu_1-\mu_1'$. 
First, we factor out the lowest common exponent to see
$$\left|e^{it\mu_0}+e^{it\mu_1}-e^{it\mu_0'}-e^{it\mu_1'} \right| =\left|1+e^{it(\mu_1 - \mu_0)}-e^{it(\mu_0' - \mu_0)}-e^{it(\mu_1'- \mu_0)} \right|.$$
Again, we denote $\phi_1=\mu_1-\mu_0$, $\phi_0'=\mu_0'-\mu_0$ and $\phi_1'=\mu_1'-\mu_0$. The following holds:
 
\begin{align*}
\left|1+e^{it\phi_1}-e^{it\phi_0'}-e^{it\phi_1'} \right| &\geq 
   |\text{Re}(1+e^{it\phi_1}-e^{it\phi_0'}-e^{it\phi_1'} )| \\
 & =  \left| 1+\cos(t \phi_1)-\cos(t \phi_0')-\cos(t \phi_1') \right|  &&[ |z| \geq |\text{Re}(z)|]\\
  & \ge  -1-\cos(t \phi_1)+\cos(t \phi_0')\\
  &\quad +\cos(t (\phi_1-(\mu_1-\mu_1')))  && [\text{Remove }|\cdot|]\\
   & \ge  -1-\cos(t \phi_1)+\cos(t \phi_0')+\cos(t (\phi_1- \phi_0'))  && [\mu_1-\mu_1' \ge \phi_0']\\
   & \ge  \frac{t^2 (\phi_1-\phi_0')\phi_0'}{2} \ge \frac{t^2 (\delta_1-\delta_4)\delta_4}{2}.   &&[\text{{\bf Fact III} below}]     \\
\end{align*}
Recall that $\delta_1 = \max(|\mu_0-\mu_1|, |\mu_0'-\mu_1'|) $,and
$ \delta_4 = \min(\left| \mu_0'-\mu_0 \right|, \left| \mu_1'-\mu_1 \right|)$, so in the above, $\delta_4 = \mu_0'-\mu_0 = \phi_0'$.
The last line uses {\bf Fact III} that $-1-\cos(x )+\cos(y)+\cos(x-y)\geq (x-y)y/2$ for $0 \leq y \leq x \leq \pi/4$.

When $\mu_0'-\mu_0 > \mu_1-\mu_1'$ we use the same trick as in the previous claims and factor out $e^{it \mu_1}$ instead of $e^{i t \mu_0}$. In particular the following holds:
$$\left|e^{it\mu_0}+e^{it\mu_1}-e^{it\mu_0'}-e^{it\mu_1'} \right| =\left|e^{i t \mu_0-\mu_1}+1-e^{it(\mu_0' - \mu_1)}-e^{it(\mu_1'- \mu_1)} \right|.$$
Again, we denote $\phi_0=\mu_0-\mu_1$, $\phi_0'=\mu_0'-\mu_1$ and $\phi_1'=\mu_1'-\mu_1$. Note that all the $\phi$ are negative and $\phi_0 > \phi_0' > \phi_1'$. The following holds:
 
\begin{align*}
\left|e^{it\phi_0}+1-e^{it\phi_0'}-e^{it\phi_1'} \right| &\geq 
   |\text{Re}(e^{it\phi_0}+1-e^{it\phi_0'}-e^{it\phi_1'} )| \\
 & =  \left| \cos(t \phi_0)+1-\cos(t \phi_0')-\cos(t \phi_1') \right|  &&[|z| \geq |\text{Re}(z)|]\\
  & =  \left| \cos(t |\phi_0|)+1-\cos(t |\phi_0'|)-\cos(t |\phi_1'|) \right|  &&[\cos(\cdot) \text{ even}]\\
  & \ge  -\cos(t |\phi_0|)-1+\cos(t( |\phi_0| - |\mu_0' - \mu_0|)) &&  [\text{Remove }|\cdot|]\\ & \quad +\cos(t (|\phi_1'|))  &&\\
   & \ge  -\cos(t |\phi_0|)-1+\cos(t( |\phi_0| - |\phi_1'|))\\
   &\quad +\cos(t |\phi_1'|)  && [\mu_0'-\mu_0 \ge | \phi_1'|]\\
   & \ge  \frac{t^2 (|\phi_0|-|\phi_1'|)|\phi_1'|}{2} \ge \frac{t^2 (\delta_1-\delta_4)\delta_4}{2}.  &&     [\text{{\bf Fact III} above}]
\end{align*}
\end{proof}
\end{proof}

\subsubsection*{Proof of Lemma~\ref{lem:small_prec}}

\begin{proof}[Proof of Lemma~\ref{lem:small_prec}]
% Without loss of generality, let us assume that $\left|\mu_0-\mu_1\right| \ge \left|\mu_0'-\mu_1'\right| \ge 100 \sigma$ {\color{red}Sami: Is this assumption wlog?? not needed anyway}.
Define $\alpha$ and $\beta$ such that $\mu_0'-\mu_0 = \alpha \sigma$ and $|\mu_1-\mu_1'|= \beta \sigma$; note that by assumption $\alpha, \beta \le 2$.  For $x \in \bb{R}$, we use the notation $\widetilde x$ to denote the unique value such that $x= 2\pi k c\sigma+\widetilde{x}\sigma$, where $k \in \bb{Z}$ is a integer and $0 \leq \widetilde{x} <2\pi c$. We prove this lemma with two cases, when $\mu_1' > \mu_1$ and when $\mu_1' \leq \mu_1$. Without loss of generality, we assume that $\left|\mu_0-\mu_1\right|\ge 100 \sigma$. Also, recall our assumption on the ordering of the unknown parameters that
$\mu_0\le \min(\mu_1,\mu_0',\mu_1')$ and $\mu_0'\leq\mu_1'$.

\paragraph{Case 1 ($\mu_1'>\mu_1$):}  
Here, we will choose  $t=c\sigma$, where 
$$c= \frac{\mu_1-\mu_0}{2\pi\sigma \lfloor\frac{\mu_1-\mu_0}{80\sigma/\pi} \rfloor}.$$ 
 Then substituting in $t = 1/c\sigma$, we see that 
$$e^{itx}=  e^{it2\pi k c\sigma}e^{it\widetilde{x}\sigma}=  e^{i2\pi k }e^{i\widetilde{x}/c}=  e^{ i\widetilde{x}/c }.$$ 
From the choice of $c$ and the fact that $\lfloor x \rfloor \le x$ and $x/2 \le \lfloor x \rfloor$ for $x \ge 1$, we see $40 / \pi^2 \le c \le 80 / \pi^2$. 

As before, let $\phi_1 = \mu_1-\mu_0, \phi_0' = \mu_0'-\mu_0$ and $\phi_1' = \mu_1'-\mu_0$.
We prove that the following hold:
\begin{align*}
\widetilde{\phi}_1 &= 0 \\
  \frac{\pi^2 \alpha}{80} \le \frac{\widetilde{\phi}_0'}{c}&= \frac{\alpha}{c} \le \frac{\pi^2 \alpha}{40} \le \frac{\pi^2}{20} \\
 \frac{\pi^2 \beta}{80} \le \frac{\widetilde{\phi}_1'}{c}&= \frac{\beta}{c} \le \frac{ \pi^2 \beta}{40} \le \frac{\pi^2}{20}.
\end{align*}
We prove these statements in order. To see that $\widetilde{\phi}_1=0$, the definitions of $c$ and $\widetilde{\phi}_1$ imply that 
$\phi_1 =  2\pi\sigma k \frac{\phi_1}{2\pi\sigma\lfloor \phi_1 \pi / (80 \sigma) \rfloor} + \widetilde{\phi}_1 \sigma$, for $k = \lfloor \phi_1 \pi / (80 \sigma) \rfloor$ and $\widetilde{\phi}_1=0$. 

Next, since $\alpha \sigma = \phi_0'$, we can write $\alpha /c = 2 \pi k + \widetilde{\phi}_0' /c$, and it would follow that $\alpha/c = \widetilde{\phi}_0'/c$ if $\phi_0' < 2 \pi \sigma c = \phi_1/ \lfloor \phi_1 \pi /(80 \sigma) \rfloor$. 
Indeed this is the case, since 
$$\phi_1/ \lfloor \phi_1 \pi /(80 \sigma) \rfloor \geq 80 \sigma/ \pi >2 \sigma > \phi_0'.$$

Using the fact that $\widetilde{\phi}_1=0$, we will show $\widetilde{\phi}_1' /c = \beta/c$. 
We break up $\phi_1'$ into $\phi_1 + \mu_1'-\mu_1$, writing
$$\phi_1' = \beta \sigma + \phi_1 = \beta \sigma  + k \frac{\phi_1}{\lfloor \phi_1 \pi /(80 \sigma) \rfloor} + \widetilde{\phi}_1 \sigma=  \beta \sigma  + k \frac{\phi_1}{\lfloor \phi_1 \pi /(80 \sigma) \rfloor},$$
for $k = \lfloor \phi_1 \pi / (80 \sigma) \rfloor$. 
If $\beta < 2 \pi c$, then this choice of $k$ is correct for the definition of $\widetilde{\phi}_1'$, and it follows that $\widetilde{\phi}_1' = \beta$. This is indeed the case as $2 \pi c = \frac{\phi_1}{\sigma \lfloor  \phi_1 \pi / (80  \sigma)\rfloor} \geq 80 /\pi > \beta$.

We use our lower bounds on $\widetilde{\phi}_0'/c$ and $\widetilde{\phi}_1'/c$ and the fact that $\widetilde{\phi}_1=0$ in the following:
\begin{align*}
e^{-\frac{\sigma^2 t^2}{2}} \left |e^{i t \mu_0} + e^{i t \mu_1} - e^{i t \mu'_0}-e^{i t \mu'_1} \right | &= 
e^{-\frac{\sigma^2 t^2}{2}} \left |1+ e^{i t \phi_1} - e^{i t \phi_0'}-e^{i t \phi_1' } \right | \\
% e^{-\frac{\sigma^2 t^2}{2}} \left| 1+e^{it\phi_1}-e^{it\phi_0'}-e^{it\phi_1'} \right| 
% &\ge e^{-\frac{1}{2c^2}} \left| \text{Im}(1+e^{it\phi_1}-e^{it\phi_0'}-e^{it\phi_1'})\right| \\
& \geq e^{-\frac{1}{2c^2}} \left| \text{Im}(1+e^{i\widetilde{\phi}_1/c}-e^{i\widetilde{\phi}_0'/c}-e^{i\widetilde{\phi}_1'/c})\right| \\
& = e^{-\frac{1}{2c^2}}  \left | \sin (\widetilde{\phi}_1/c) - \sin(\widetilde{\phi}_0'/c) - \sin(\widetilde{\phi}_1'/c)\right |\\
& = e^{-\frac{1}{2c^2}} (   \sin(\widetilde{\phi}_0'/c) + \sin(\widetilde{\phi}_1'/c))\\
&\ge e^{-1} \max(\sin (\widetilde{\phi}'_0/c), \sin (\widetilde{\phi}'_1/c)) \\ 
&\ge e^{-1} \max(\sin (\pi^2\alpha/80),\sin (\pi^2\beta/80)) \\
&\ge \frac{\pi^2 \delta_2}{160e}.
\end{align*}

\paragraph{Case 2 ($\mu_1' \leq \mu_1$):}  

First we consider the case when $\beta \leq \alpha$. Since $\mu_1-\mu_0 \ge 100\sigma$, we must have $\mu_1-\mu_0 \ge \mu_1-\mu_0-(\mu_1-\mu_1')\ge 100\sigma -2\sigma =98\sigma$. We choose $t=c\sigma$ for 
$$c= \frac{\mu_1'-\mu_0}{3\pi\sigma/2+2\pi\sigma \lfloor\frac{\mu_1'-\mu_0}{80\sigma/\pi} \rfloor}.$$
As before, for any $x \in \bb{R}$ we write $x= 2\pi k c\sigma+\widetilde{x}\sigma$ where $k \in \bb{Z}$ is a positive integer and $0<\widetilde{x} <2\pi c$. 
% Therefore we will have $e^{itx}= e^{ \frac{i\widetilde{x}}{c} }$ on substituting $t = 1/c\sigma$.
From the choice of $c$ and the fact that $\mu_1' - \mu_0 \geq 98 \sigma$, $\frac{25}{\pi^2} \le c \le \frac{80}{\pi^2}$.
% from the choice of $c$ and by using the fact that $x \ge \lfloor x \rfloor$ and $x/2 \le \lfloor x \rfloor$ for $x \ge 1$.
Here we denote $\phi_1 = \mu_1-\mu_1', \phi_0' = \mu_0'-\mu_0$ and $\phi_1' = \mu_1'-\mu_0$; note this is different than our previous $\phi$ definitions.
We will show the following set of inequalities and equalities:
\begin{align*}
\frac{\widetilde{\phi}_1'}{c} &= \frac{3\pi}{2}  \\
\frac{\pi^2 \alpha}{80} \le \frac{\widetilde{\phi}_0'}{c} &= \frac{\alpha}{c} \le \frac{\pi^2 \alpha}{25} \le \frac{\pi^2}{12} \\
  \frac{\pi^2 \beta}{80} \le \frac{\widetilde{\phi}_1}{c}& = \frac{\beta}{c} \le \frac{ \pi^2 \beta}{25} \le \frac{\pi^2}{12}.
\end{align*}
To see that $\widetilde{\phi}'_1=3 \pi/2$,
% note from the definition of $\widetilde{\phi}'_1$ that 
% $\phi'_1 =  k \frac{\phi'_1}{3 /2 +\lfloor \phi_1' \pi / (80 \sigma) \rfloor} + \widetilde{\phi}_1 \sigma$ and so one can take $k = \lfloor \phi \pi / (80 \sigma) \rfloor$, 
observe first that $\phi_1'/ (c \sigma )= 2 \pi  k  + \widetilde{\phi}_1' /c$; then we can simplify  $\phi_1'/(c \sigma)$ and write
$
\phi_1'/(c \sigma) = 3 \pi /2 + 2 \pi \left \lfloor \phi_1' \pi/ (80 \sigma) \right \rfloor.$
Together these imply that $3 \pi /2 + 2 \pi \left \lfloor \phi_1' \pi/ (80 \sigma) \right \rfloor = 2 \pi  k  + \widetilde{\phi}_1' /c$. 
Taking $k = \left \lfloor \phi_1' \pi/ (80 \sigma) \right \rfloor $, it follows that  $\widetilde{\phi}'_1=3 \pi/2$.

Additionally, since $\alpha \sigma = \phi_0'$, we can write $\alpha /c = 2 \pi k + \tilde{\phi}_0' /c$. It follows that $\alpha/c = \tilde{\phi}_0'/c$ if 
$$\phi_0' < 2 \pi \sigma c =2 \pi \sigma \frac{\phi_1'}{3 \pi \sigma/2 + 2 \pi \sigma \lfloor \phi_1' \pi / (80 \sigma) \rfloor} 
= \frac{\phi_1'}{3 /4 +  \lfloor \phi_1' \pi /(80 \sigma )\rfloor} .$$
Indeed this is the case, since if $\lfloor \phi'_1 \pi /(80 \sigma) \rfloor < 1/4$,
$$\frac{\phi_1'}{ 3 /4 +  \lfloor \phi_1' \pi / (80 \sigma) \rfloor} > \phi_1' > \phi_0'$$
and if $\lfloor \phi'_1 \pi /(80 \sigma) \rfloor \geq 1/4$,
$$\frac{\phi_1'}{ 3 /4 +  \lfloor \phi_1' \pi / (80 \sigma) \rfloor} > \frac{\phi_1'}{  4 \lfloor \phi_1' \pi / (80 \sigma) \rfloor}> 80 \sigma /(4 \pi)  > 6 \sigma > \phi_0'.$$
A similar line of reasoning shows that $\widetilde{\phi}_1 /c = \beta/c$. Here $\beta /c = 2 \pi k + \widetilde{\phi}_1/c$, so it remains to show  
$$\phi_1 < 2 \pi \sigma c 
= \frac{\phi_1'}{3 /4 +  \lfloor \phi_1' \pi /(80 \sigma )\rfloor} .$$
Indeed this is the case, since if $\lfloor \phi'_1 \pi /(80 \sigma) \rfloor < 1/4$,
$$\frac{\phi_1'}{ 3 /4 +  \lfloor \phi_1' \pi / (80 \sigma) \rfloor} > \phi_1'  > 98 \sigma > 2 \sigma > \mu_1 - \mu_1' = \phi_1,$$
and  if $\lfloor \phi'_1 \pi /(80 \sigma) \rfloor \geq 1/4$, then
$$\frac{\phi_1'}{ 3 /4 +  \lfloor \phi_1' \pi / (80 \sigma) \rfloor} > \frac{\phi_1'}{  4 \lfloor \phi_1' \pi / (80 \sigma) \rfloor}> 80 \sigma /(4 \pi)  > 6 \sigma > \phi_1.$$

Setting $t=1/c\sigma$, the following calculation holds if 
$\beta \leq \alpha$:
\begin{align*}
e^{-\frac{\sigma^2 t^2}{2}} \left |e^{i t \mu_0} + e^{i t \mu_1} - e^{i t \mu'_0}-e^{i t \mu'_1} \right | &= 
e^{-\frac{\sigma^2 t^2}{2}} \left| 1+e^{it(\phi_1 + \phi_1')}-e^{it\phi_0'}-e^{it\phi_1'} \right| \\
 &=e^{-\frac{1}{2c^2}} \left|1+e^{i\frac{\widetilde{\phi}_1}{c}}e^{i\frac{\widetilde{\phi}_1'}{c}}-e^{i\frac{\widetilde{\phi}_0'}{c}}-e^{i\frac{\widetilde{\phi}_1'}{c}} \right| \\
% &= e^{-\frac{1}{2c^2}} \left| 1+e^{i\frac{\mu_1-\mu_1'}{c}}e^{i\frac{\widetilde{\phi}_1'}{c}}-e^{i\frac{\widetilde{\phi}_0'}{c}}-e^{i\frac{\widetilde{\phi}_1'}{c}} \right| \\
& \geq e^{-\frac{1}{2c^2}} \left|\textrm{Im}(1+e^{i\frac{\widetilde{\phi}_1}{c}}e^{i\frac{\widetilde{\phi}_1'}{c}}-e^{i\frac{\widetilde{\phi}_0'}{c}}-e^{i\frac{\widetilde{\phi}_1'}{c}})\right| \\
% & = e^{-\frac{1}{2c^2}} \left| -\cos \frac{\widetilde{\phi}_1}{c} -\sin \frac{\tilde{\phi}_0'}{c}+i\right|\\
% &= e^{-\frac{1}{2c^2}} \left|\textrm{Im}( 1-i\cos \frac{\tilde{\phi}_1}{c}+\sin \frac{\tilde{\phi}_1}{c}-\cos \frac{\tilde{\phi}_2}{c}-i\sin \frac{\tilde{\phi}_2}{c}+i  )  \right| \\
&\ge e^{-\frac{1}{c^2}} (-1+\cos \frac{\tilde{\phi}_1}{c}+\sin \frac{\tilde{\phi}_0'}{c}) \\
&= e^{-\frac{1}{2c^2}} (-1+\cos \frac{\beta}{c}+\sin \frac{\alpha}{c}) \\
& \ge e^{-\frac{1}{2c^2}} ( \alpha/c-(\alpha/c)^3/6-(\beta/c)^2/2) \\
&\ge e^{-\frac{1}{2c^2}} (\alpha/c-(\alpha/c)^3/6-(\alpha/c)^2/2)\\
&\ge e^{-\frac{1}{2c^2}} \frac{\alpha}{3c}  
 \ge \frac{\pi^2 \delta_2}{240e}.
% &= e^{-\frac{1}{c^2}} \frac{\max(\mu_1-\mu_1',\mu_0'-\mu_0)}{2c} \ge \frac{\pi^2 \delta_2}{160e\sigma}
\end{align*}
The fourth to last inequality follows because $\sin x \ge x-\frac{x^3}{6}$ and $\cos x \ge 1- \frac{x^2}{2}$. The third to last inequality follows because $\beta \leq \alpha$ and $\alpha/c < 1$. In the final step, we re-used the fact that $\frac{25}{\pi^2} \le c \le \frac{80}{\pi^2}$.

If $\alpha < \beta$, then we can do a very similar proof by choosing 
 $t=c\sigma$ for 
$$c= \frac{\mu_1-\mu_0'}{3\pi\sigma/2+2\pi\sigma \lfloor\frac{\mu_1-\mu_0'}{80\sigma/\pi} \rfloor}.$$
From the choice of $c$ and the fact that $\mu_1 - \mu_0' \geq 98 \sigma$, $\frac{25}{\pi^2} \le c \le \frac{80}{\pi^2}$.
% from the choice of $c$ and by using the fact that $x \ge \lfloor x \rfloor$ and $x/2 \le \lfloor x \rfloor$ for $x \ge 1$.
Here we denote $\phi_1' = \mu_1'-\mu_1, \phi_0' = \mu_0'-\mu_1$ and $\phi_0 = \mu_0- \mu_0'$.
From the same explanations as in the case when $\beta \leq \alpha$, we see that
\begin{align*}
\frac{|\widetilde{\phi}_0'|}{c} &= \frac{3\pi}{2}  \\
\frac{\pi^2 \beta}{80} \le \frac{|\widetilde{\phi}_1'|}{c} &= \frac{\beta}{c} \le \frac{\pi^2 \beta}{25} \le \frac{\pi^2}{12} \\
  \frac{\pi^2 \alpha}{80} \le \frac{|\widetilde{\phi}_0|}{c}& = \frac{\alpha}{c} \le \frac{ \pi^2 \alpha}{25} \le \frac{\pi^2}{12}.
\end{align*}

We obtain the same bound as in the case of $\beta \leq \alpha$ by factoring out $e^{i t \mu_1}$ and using a similar calculation:
\begin{align*}
e^{-\frac{\sigma^2 t^2}{2}} \left |e^{i t \mu_0} + e^{i t \mu_1} - e^{i t \mu'_0}-e^{i t \mu'_1} \right | &= 
e^{-\frac{\sigma^2 t^2}{2}} \left|e^{i t (\mu_0 - \mu_1)} + 1 - e^{i t (\mu'_0 -\mu_1)}-e^{i t (\mu'_1 - \mu_1)}\right| \\
&= 
e^{-\frac{\sigma^2 t^2}{2}} \left|e^{i t (\phi_0+\phi_0')} + 1 - e^{i t \phi_0'}-e^{i t \phi_1'}\right| \\
 &=e^{-\frac{1}{2c^2}} \left|1+e^{i\frac{\widetilde{\phi}_0}{c}}e^{i\frac{\widetilde{\phi}_0'}{c}}-e^{i\frac{\widetilde{\phi}_0'}{c}}-e^{i\frac{\widetilde{\phi}_1'}{c}} \right| \\
& \geq e^{-\frac{1}{2c^2}} \left|\textrm{Im}\left(1+e^{i\frac{-|\widetilde{\phi}_0|}{c}}e^{i\frac{-|\widetilde{\phi}_0'|}{c}}-e^{i\frac{-|\widetilde{\phi}_0'|}{c}}-e^{i\frac{-|\widetilde{\phi}_1'|}{c}}\right)\right| \\
% & = e^{-\frac{1}{2c^2}} \left|\textrm{Im}(e^{i\frac{-|\tilde{\phi}_0|}{c}}e^{-i3 \pi/2}-e^{-i3 \pi/2}-e^{i\frac{-|\tilde{\phi}_1'|}{c}})\right| \\
% & = e^{-\frac{1}{2c^2}} \left|\textrm{Im}(e^{i\frac{-|\tilde{\phi}_0|}{c}}i-i-e^{i\frac{-|\tilde{\phi}_1'|}{c}})\right| \\
% & = e^{-\frac{1}{2c^2}} \left|\textrm{Im}(\cos(\frac{-|\tilde{\phi}_0|}{c})i-i-\sin(i\frac{-|\tilde{\phi}_1'|}{c})\right| \\
% & = e^{-\frac{1}{2c^2}} \left|\textrm{Im}(\cos(\frac{|\tilde{\phi}_0|}{c})i-i+\sin(\frac{|\tilde{\phi}_1'|}{c})\right| \\
&\ge e^{-\frac{1}{c^2}} \left(-1+\cos \frac{|\tilde{\phi}_0|}{c}+\sin \frac{|\tilde{\phi}_1'|}{c}\right) \\
&= e^{-\frac{1}{c^2}} \left(-1+\cos \frac{\alpha}{c}+\sin \frac{\beta}{c}\right),
% & \ge e^{-\frac{1}{c^2}} ( \beta/c-(\beta/c)^3/6-(\alpha/c)^2/2) \\
% &\ge e^{-\frac{1}{c^2}} (\alpha/c-(\alpha/c)^3/6-(\alpha/c)^2/2)\\
% &\ge e^{-\frac{1}{c^2}} \frac{\alpha}{3c}  
%  \ge \frac{\pi^2 \delta_2}{240e}
% &= e^{-\frac{1}{c^2}} \frac{\max(\mu_1-\mu_1',\mu_0'-\mu_0)}{2c} \ge \frac{\pi^2 \delta_2}{160e\sigma}
\end{align*}
and the rest of the proof follows as before, just swapping $\alpha$ and $\beta$.
\end{proof}

\end{document}